\documentclass[11pt,british]{amsart}

\usepackage{babel,dutchcal,bm,xcolor,amssymb,microtype,overpic}
\definecolor{strongblue}{RGB}{0,0,160}
\definecolor{strongred}{RGB}{160,0,0}
\renewcommand*{\eqref}[1]{\textcolor{strongblue}{(\ref{#1})}}
\usepackage[ocgcolorlinks=true,citecolor=strongblue,linkcolor=strongblue,urlcolor=strongred,backref=page]{hyperref}
\usepackage[marginratio=1:1,paperwidth=460pt,paperheight=700pt,height=590pt,width=360pt]{geometry}
\usepackage[capitalize,nameinlink,noabbrev]{cleveref}
\usepackage{comment}

\date{27th November 2020}

\newcommand*{\imag}{\mathrm{i}}
\newcommand*{\rot}{R_{\mathrm{e}^{\imag\vartheta}}}

\renewcommand*{\d}{\mathrm{d}}
\newcommand*{\R}{\mathbb{R}}
\newcommand*{\C}{\mathbb{C}}

\newcommand*{\f}{\bm{B}}
\newcommand*{\ov}{\overline}
\newcommand*{\un}{\underline}
\newcommand*{\D}{\mathrm{D}}
\newcommand{\V}{\mathbb{V}}
\newcommand*{\F}{F}
\renewcommand*{\Im}{\operatorname{Im}}
\renewcommand*{\Re}{\operatorname{Re}}
\newcommand*{\inc}{\hspace{3pt}\rule[0.5pt]{2mm}{0.5pt}\rule[0.5pt]{0.5pt}{4.5pt}\hspace{3pt}}
\newcommand*{\hmet}{\mathcal{h}}
\newcommand*{\h}{h}

\title[Vortices and Dominated Splittings]{Vortices over Riemann Surfaces\\ and Dominated Splittings}
\author[T.~Mettler]{Thomas Mettler}
\address{Institut f\"ur Mathematik, Goethe-Universit\"at Frankfurt, 60325 Frankfurt am Main, Germany}
\email{mettler@math.uni-frankfurt.de, mettler@math.ch}
\author[G.P.~Paternain]{Gabriel P.~Paternain}
\address{Department of Pure Mathematics and Mathematical Statistics,
University of Cambridge,
Cambridge CB3 0WB, England}
\email{g.p.paternain@dpmms.cam.ac.uk}

\newtheorem{Theorem}{Theorem}[section]
\newtheorem{Lemma}[Theorem]{Lemma}

\newtheorem{Proposition}[Theorem]{Proposition}

\newtheorem{mainthm}{Theorem}

\newtheorem{introthm}{Theorem}

\theoremstyle{definition}
\newtheorem{Definition}[Theorem]{Definition}
\theoremstyle{remark}
\newtheorem{Remark}[Theorem]{Remark}
\newtheorem{Example}[Theorem]{Example}

\numberwithin{equation}{section}

\begin{document}

\begin{abstract}
We associate a flow $\phi$ to a solution of the vortex equations on a closed oriented Riemannian 2-manifold $(M,g)$ of negative Euler characteristic and investigate its properties. We show that $\phi$ always admits a dominated splitting and identify special cases in which $\phi$ is Anosov. In particular, starting from holomorphic differentials of fractional degree, we produce novel examples of Anosov flows on suitable roots of the unit tangent bundle of $(M,g)$.  
\end{abstract}

\maketitle

\section{Introduction}

\subsection{Background}

This paper is concerned with the description and study of a class of dynamical systems determined by the solutions of a pair of partial differential equations naturally arising in Abelian gauge theories on a closed oriented Riemannian $2$-manifold $(M,g)$ of negative Euler characteristic. 

Let $SM$ denote the unit tangent bundle of $(M,g)$. Given a smooth function $\lambda\in C^{\infty}(SM)$, we may consider the ODE for  $\gamma:\mathbb{R}\to M$
\begin{equation}
\ddot{\gamma}=\lambda(\gamma,\dot{\gamma})J\dot{\gamma},
\label{eq:thermo}
\end{equation}
where $J:TM\to TM$ denotes rotation by $\pi/2$ according to the orientation of the surface, and the acceleration of $\gamma$ is computed using the Levi-Civita connection of $g$. Equation \eqref{eq:thermo} describes the motion of particle on $M$ driven by a force orthogonal to its velocity with magnitude determined by $\lambda$.
As such it is easy to see that the speed of $\gamma$ remains constant and thus $(\gamma,\dot{\gamma})$ defines
a flow in $SM$. If $\lambda=0$, we obtain the geodesic flow of $g$, the prototype example of a conservative dynamical system. If $\lambda$ only depends on position (i.e. it is the pull-back of a function on $M$), we still obtain a volume preserving flow (a magnetic flow), but the situation changes if $\lambda$ is allowed to depend on velocities. For instance we may take $\lambda$ as the restriction to $SM$ of a 1-form on $M$ and in that case
we obtain a {\it Gaussian thermostat} as studied in \cite{W1,W2}. In general, these flows are not volume preserving and here we are concerned with thermostat flows as defined by \eqref{eq:thermo} when
$\lambda$ arises from a higher order differential on $M$.

The dynamical properties that we shall investigate are {\it hyperbolicity} and {\it domination}.  Hyperbolicity has played a prominent role in dynamics \cite{MR228014}, but weaker forms of hyperbolicity, like domination have in recent decades come under intense focus \cite{MR2105774}. The notion of dominated splitting was introduced
by Ma\~n\'e in the context of the proof of the stability conjecture (cf.~\cite{MR3457601}), but it has appeared in several other contexts and under different names. It can be regarded as a projective form of hyperbolicity
and it can also be characterized in terms of the singular value decomposition of the linear Poincar\'e flow \cite{MR2529495}.
The notion is particularly relevant in our setting: for volume preserving flows on 3-manifolds domination is equivalent to hyperbolicity, but for dissipative thermostats this is no longer the case.  Thus in the results below
some effort will be spent in studying when we can upgrade our flows from having a dominatted splitting to being
Anosov.

Let us give the benchmark example that motivates
our construction. Let $A$ be a {\it holomorphic} cubic differential on $M$ so that $\ov{\partial}A=0$, and suppose the pair $(g,A)$ is linked
by the additional equation
\[K_{g}=-1+2|A|^2_{g},\]
where $K_{g}$ is the Gauss curvature. By the work of Labourie~\cite{MR2402597} and Loftin \cite{MR1828223}, such a pair gives rise to a properly convex projective structure on $M$ and hence to an associated divisible strictly convex set $\tilde{M}\subset \mathbb{RP}^2$. The set $\tilde{M}$ comes equipped with a distance function, the so-called \emph{Hilbert metric} --- see for instance~\cite{MR3329885} for details --- while $g$ is known as the \emph{Blaschke metric}. The Hilbert metric is the distance function of a Finsler metric whose geodesic flow is known to be Anosov~\cite{MR2094116}. If we choose $\lambda$ to be the imaginary part of $A$ --- regarded as a function on $SM$ --- then the thermostat flow determined by \eqref{eq:thermo} is a suitable reparametrization of the geodesic flow of the Hilbert metric. While the work of Labourie interprets the pair
of equations $\ov{\partial}A=0$ and $K_{g}=-1+2|A|^2_{g}$ as an instance of Hitchin's Higgs bundle equations~\cite{MR887284}, they may also be interpreted as an example of the so-called Abelian vortex equations~\cite{MR2911018}. One can, in fact, consider similar equations for differentials of any order, not just 3, and investigate the dynamical properties of the associated thermostat. This was done in \cite{GabrielThomas_Thermo}, but here we uncover a larger landscape that allows for example the consideration of holomorphic differential of {\it fractional order}, i.e. holomorphic sections of $K^{m/n}$ where $K$ is the canonical line bundle of $(M,g)$.
The natural habitat of our thermostats is not the unit sphere bundle anymore, but rather {\it root bundles}
covering $SM$ to accommodate for the fractional degrees.

\subsection{Vortices} We now proceed to describe in detail the geometric setting for our pair
of PDEs.

Let $L \to M$ be a complex line bundle of positive degree $\deg(L)$. For a triple consisting of a Hermitian bundle metric $\hmet$ on $L$, a del-bar operator $\ov{\partial}_L$ on $L$ and a $(1,\! 0)$-form $\varphi$ on $M$ with values in $L$, we consider the following pair of equations
\begin{equation}\label{eq:vortexintro}
R(\D)+\frac{1}{2}\varphi\wedge\varphi^*+\imag \ell\Omega_g=0 \qquad \text{and} \qquad \ov{\partial}_L\varphi=0. 
\end{equation}
Here we write $\ell:=\deg(L)/|\chi(M)|$, $\D$ denotes the Chern connection on $L$ with respect to $(\hmet,\ov{\partial}_L)$, $R(\D)$ its curvature, $\Omega_g$ the area form of $g$ and $\varphi^*:=\hmet(\cdot,\varphi)$. We assume $\hmet$ to be conjugate linear in the second variable, so that $\varphi^*$ is a $(0,\! 1)$-form on $M$ with values in the dual $L^{-1}$ of $L$. 

The pair~\eqref{eq:vortexintro} of equations are a minor variation of the Abelian vortex equations on a Riemann surface, hence we refer to them as~\textit{vortex equations} as well. The usual Abelian vortex equations concern a triple $(\mathcal{h},\ov{\partial}_{L^{\prime}},\Phi)$, where $\Phi$ is a section of a complex line bundle $L^{\prime}$ over an oriented Riemannian $2$-manifold $(M,g)$. Besides $\Phi$ being holomorphic, one requests that the Chern connection $\D$ determined by $(\mathcal{h},\ov{\partial}_{L^{\prime}})$ satisfies
\begin{equation}\label{eq:usualvortex}
\imag\Lambda R(\D)+\frac{1}{2}\Phi\otimes \Phi^*-\frac{c}{2}=0,
\end{equation}
where $c$ is some real constant and $\Lambda$ denotes the $L^2$-adjoint of wedging with the area form $\Omega_g$. The Abelian vortex equations are a modification of the Ginzburg--Lan\-dau model for superconductors and were first studied by Noguchi~\cite{MR907998} and Bradlow~\cite{MR1086749} (for background, see also~\cite{MR614447}). A general framework for the so-called symplectic vortices over closed Riemann surfaces was described in~\cite{MR1959059}.   

\subsection{Vortex thermostats} Since $L$ has positive degree and $\chi(M)<0$, there exist unique positive coprime integers $(m,n)$ so that we have an isomorphism $L^n\simeq K^m$ of complex line bundles. We fix an $n$-th root $SM^{1/n}$ of the unit tangent bundle $\pi : SM \to M$ of $(M,g)$. By this we mean a principal $\mathrm{SO}(2)$-bundle $\pi_n : SM^{1/n} \to M$ which is an equivariant $n$-fold cover of $\pi: SM \to M$, see~\cref{subsec:roots} below for details. 

Following~\cite{Bry}, we call three linearly independent vector fields $(X,H,V)$ on a smooth $3$-manifold $N$ a~\textit{generalised Riemannian structure}, if they satisfy the commutator relations 
\[
[V,X]=H,\qquad [V,H]=-X, \qquad [X,H]=K_g V,\]
for some smooth function $K_g$ on $N$. A~\textit{(generalised) thermostat} is a flow $\phi$ on $N$ which is generated by a vector field of the form $X+\lambda V$, where $\lambda$ is a smooth function on $N$. The root $SM^{1/n}$ is equipped with a generalised Riemannian structure by pulling back the natural Riemannian structure on $SM$ determined by $g$ and the orientation (where $X$ is the geodesic vector field and $V$ the vertical vector field).  In~\cref{sec:vortextothermo} we show how to associate a thermostat to a solution $(\hmet,\ov{\partial}_L,\varphi)$ of the vortex equations on $L \to (M,g)$. We call such flows~\textit{vortex thermostats}. 

In the special case where $L$ is the canonical bundle equipped with its standard complex structure and Hermitian metric induced by $g$ and where $\varphi$ vanishes identically, the vortex equations~\eqref{eq:vortexintro} are equivalent to $g$ being hyperbolic. The case where $g$ has non-constant negative Gauss curvature can be dealt with by modifying the complex structure on $K$. In particular, suitably reparametrised, our family of flows include the geodesic flow of metrics of negative Gauss curvature and more generally the so-called W-flows of Wojtkowski~\cite{W1,W2}  (in the case of negative curvature, c.f.~\cite[Remark 4.10]{GabrielThomas_Thermo}).

\subsection{Results} Our goal is to establish hyperbolicity properties for the general class of vortex thermostats.
We first show:

\begin{introthm}\label{mainthmintro:domsplit}
Every vortex thermostat admits a dominated splitting. More\-over, if all closed orbits of $\phi$ are hyperbolic saddles, then $\phi$ is Anosov.
\end{introthm}

The choice of an $n$-th root $SM^{1/n}$ of $SM$ gives a corresponding $n$-th root $K^{1/n}$ of $K$ and hence an isomorphism $\mathcal{Z} : L \to  K^{m/n}$ of complex line bundles. While $\mathcal{Z}$ is in general  not an isomorphism of holomorphic line bundles, we can upgrade~\cref{mainthmintro:domsplit} as follows:
\begin{introthm}\label{mainthmintro:anosov}
Suppose $\mathcal{Z} : L \to K^{m/n}$ is an isomorphism of holomorphic line bundles, then the associated vortex thermostat is Anosov. 
\end{introthm}
We do not know if there is a vortex thermostat which is not Anosov.

As in the case of the usual vortex equations, the equations~\eqref{eq:vortexintro} are invariant under a suitable action of the complex gauge group of $L$, that is, the group $\mathrm{G}_{\C}$ of automorphisms of $L$. We show that by possibly applying a complex gauge transformation, we can assume without losing generality that $\hmet=\hmet_0$, where $\hmet_0$ denotes the natural Hermitian bundle metric on $L\simeq K^{m/n}$ determined by $g$. The $1$-form $\varphi$ is a section of $K\otimes L\simeq K^{1+\ell}$ and hence we may think of $\varphi/\ell$ as a differential $A$ of fractional degree $1+\ell>1$. Furthermore, since $K^{m/n}\simeq L$ as complex line bundles, there exists a unique $1$-form $\theta$ on $M$ so that $\ov{\partial}_{K^{m/n}}-\ov{\partial}_L=\ell\theta^{0,1}$, where $\theta^{0,1}$ denotes the $(0,\! 1)$-part of $\theta$. By construction, the above isomorphism $\mathcal{Z}$ of complex line bundles is an isomorphism of holomorphic line bundles if and only if $\theta$ vanishes identically. In terms of the triple $(g,A,\theta)$ the vortex equations~\eqref{eq:vortexintro} are equivalent to
\[
K_g-\delta_g\theta=-1+\ell |A|^2_g \qquad \text{and}\qquad \ov{\partial}A=\ell\,\theta^{0,1}\otimes A,
\]
where $|\cdot|_g$ denotes the pointwise norm induced on $K^{1+\ell}$ by $g$ and $\delta_g$ the co-differential. Thus, we recover the main equations from~\cite{GabrielThomas_Thermo} (see also~\cite{MR3987443}), but now in the more general setting of fractional differentials. In particular, \cref{mainthmintro:domsplit} and \cref{mainthmintro:anosov} above generalise the results from~\cite{GabrielThomas_Thermo} to the case of differentials of fractional degree.  
Proving the Anosov property for fractional differential presents new obstacles, particularly those in the range $0<\ell< 1$. 

As in~\cite{GabrielThomas_Thermo}, our flows do not preserve a volume form, unless $\varphi$ vanishes. More precisely, the proof of \cite[Theorem 5.5]{GabrielThomas_Thermo} shows that under the hypotheses of \cref{mainthmintro:anosov} the associated vortex thermostat $\phi$ preserves an absolutely continuous measure if and only if $\varphi$ vanishes identically. This property implies that vortex thermostats as in \cref{mainthmintro:anosov} with $\varphi\neq 0$ have positive entropy production and thus they provide  interesting models in nonequilibrium statistical mechanics
\cite{MR1812682,MR1489572,MR1705592}. The moduli space of gauge equivalence classes of solutions of the usual vortex equations was  described in~\cite[Theorem 4.6]{MR1086749}, we expect a similar statement to hold as well in the case considered here; this may be taken up elsewhere. 

In~\cref{app:gencase} we briefly discuss the dominated splitting property for a thermostat that one can associate to the usual vortex equations.  

\subsection*{Acknowledgements} We are grateful to Miguel Paternain, Rafael Potrie and the anonymous referee for helpful comments. TM was partially funded by the priority programme SPP 2026 ``Geometry at Infinity'' of DFG. GPP was partially supported by EPSRC grant EP/R001898/1.

\section{Preliminaries}

\subsection{The unit tangent bundle}

Let $(M,g)$ be an oriented Riemannian $2$-manifold and let $\pi : SM \to M$ denote its unit tangent bundle. Recall that $SM$ is equipped with a coframing consisting of three linearly independent $1$-forms $(\un{\omega_1},\un{\omega_2},\un{\psi})$. The $1$-forms $(\un{\omega_1},\un{\omega_2})$ span the $1$-forms on $SM$ that are semibasic for the basepoint projection $\pi$, that is, the forms that vanish when evaluated on vertical vector fields. Explicitly, we have for all $(x,v) \in SM$ and $\xi \in T_{(x,v)}SM$
\[\un{\omega_1}(\xi)=g(d\pi(\xi),v)\quad\text{and}\quad \un{\omega_2}(\xi)=g(d\pi(\xi),Jv),\]
where $J : TM \to TM$ denotes rotation by $\pi/2$ in counter-clockwise direction with respect to the fixed orientation. The third $1$-form $\un{\psi}$ is the Levi-Civita connection form of $g$ so that we have the structure equations
\[\d\un{\omega_1}=-\un{\omega_2}\wedge\un{\psi},\qquad \d\un{\omega_2}=-\un{\psi}\wedge\un{\omega_1},\qquad \d\un{\psi}=-K_g\un{\omega_1}\wedge\un{\omega_2},\]
where $K_g$ denotes the (pullback to $SM$ of the) Gauss curvature of $g$. Denoting by $(\un{X},\un{H},\un{V})$ the vector fields dual to $(\un{\omega_1},\un{\omega_2},\un{\psi})$, the structure equations imply the commutator relations
\begin{equation}\label{eq:genstruceq}
[\un{V},\un{X}]=\un{H},\qquad [\un{V},\un{H}]=-\un{X}, \qquad [\un{X},\un{H}]=K_g\un{V}.
\end{equation}
The vector field $\un{X}$ is the geodesic vector field of $(M,g)$ and $\un{V}$ is the generator of the $\mathrm{SO}(2)$ right action on $SM$ which we denote by $R_{\mathrm{e}^{\imag\vartheta}}$ for $\mathrm{e}^{\imag\vartheta} \in \mathrm{SO}(2)$. 

Note that a complex-valued $1$-form on $M$ that is a $(1,\! 0)$-form with respect to the Riemann surface structure defined by $J$ pulls back to $SM$ to become a complex multiple of the form $\un{\omega}:=\un{\omega_1}+\imag \un{\omega_2}$. The form $\un{\omega}$ satisfies the equivariance property $(\rot)^*\un{\omega}=\mathrm{e}^{-\imag\vartheta}\un{\omega}$ for all $\mathrm{e}^{\imag\vartheta} \in \mathrm{SO}(2)$ and hence a section $\beta$ of the canonical bundle $K$ of $M$ is represented by a complex-valued function $\bm\beta$ on $SM$ satisfying the equivariance property $(\rot)^*\bm\beta=\mathrm{e}^{\imag\vartheta}\bm\beta$. To recover the associated $(1,\! 0)$-form on $M$, we observe that $\bm\beta\un{\omega}$ is semi-basic and invariant under the $\mathrm{SO}(2)$-right action, hence the pullback of a unique $(1,\! 0)$-form on $M$, which is $\beta$. 

\begin{Remark}[Notation]
We write $Y(f)$ for the (Lie-)derivative of a smooth real -- or complex-valued function $f$ in the direction of a vector field $Y$. Whenever no confusion is possible about the argument of the linear differential operator $Y$, we will simply write $Yf$ instead of $Y(f)$.  
\end{Remark} 

\subsection{Roots of the unit tangent bundle}\label{subsec:roots}

Let $n \in \mathbb{N}$ and $\pi_{n} : SM^{1/n} \to M$ be a principal right $\mathrm{SO}(2)$-bundle whose right action we denote by $R_{\mathrm{e}^{\imag\vartheta}}$ as well. Let $\pi : SM \to M$ denote the unit tangent bundle of the oriented Riemannian $2$-manifold $(M,g)$ and $(\un{\omega_1},\un{\omega_2},\un{\psi})$ its coframing. We call $\pi_{n} : SM^{1/n} \to M$ an~\textit{$n$-th root of $SM$} if there exists an $n$-fold covering map $\rho : SM^{1/n} \to SM$ so that $\pi_{n}=\pi \circ \rho$ and so that
\[
\rho\circ R_{\mathrm{e}^{\imag\vartheta}}=R_{\mathrm{e}^{\imag n\vartheta}}\circ \rho
\]
for all $\mathrm{e}^{\imag\vartheta} \in \mathrm{SO}(2)$. We refer the reader to~\cite{MR2945757} for background about $n$-th roots of $SM$. We write $\omega_i=\rho^*\un{\omega_i}$ and $\psi=\rho^*\un{\psi}$ and let $(X,H,\V)$ denote the framing dual to $(\omega_1,\omega_2,\psi)$ on $SM^{1/n}$. The structure equations imply the usual commutator relations
\begin{equation}\label{eq:comrelation}
[\V,X]=H,\qquad [\V,H]=-X, \qquad [X,H]=K_g\V.
\end{equation}
Recall that a section $\beta$ of the canonical bundle $K$ of $(M,g)$ is represented by a complex-valued function $\bm\beta$ on $SM$ satisfying the equivariance property $(\rot)^* \bm\beta=\mathrm{e}^{\imag\vartheta} \bm\beta$. Writing $\tilde{ \bm\beta}:=\bm\beta\circ \rho$, the function $\tilde{\bm\beta}$ satisfies $\rot^*\tilde{\bm\beta}=\mathrm{e}^{\imag n\vartheta}\tilde{\bm\beta}$ 
and hence we obtain a $n$-th root $K^{1/n}$ of $K$ whose sections are represented by complex-valued functions $\f$ on $SM^{1/n}$ satisfying $(\rot)^* \f=\mathrm{e}^{\imag\vartheta} \f$ for all $\mathrm{e}^{\imag\vartheta} \in \mathrm{SO}(2)$. 
Likewise, for each $m \in \mathbb{Z}$, the smooth sections of $K^{m/n}$ are represented by smooth complex-valued functions $\f$ on $SM^{1/n}$ satisfying 
\begin{equation}\label{eq:rightactionup}
(\rot)^* \f=\mathrm{e}^{\imag m\vartheta} \f
\end{equation}
for all $\mathrm{e}^{\imag\vartheta} \in \mathrm{SO}(2)$. In particular, for each $m \in \mathbb{Z}$ we obtain a Hermitian bundle metric $\hmet_0$ on $K^{m/n}$ defined by
\[
(\f_1,\f_2) \mapsto \f_1\ov{\f_2},
\]
where $\f_1,\f_2$ represent sections of $K^{m/n}$.  

Furthermore, observe that by definition, $\V$ is only $(1/n)$-th of the generator $V$ of the $\mathrm{SO}(2)$-action on $SM^{1/n}$. As a consequence, the infinitesimal version of~\eqref{eq:rightactionup} becomes
\begin{equation}\label{eq:rightactionupinf}
\V\f=\frac{1}{n}V\f=\imag\left(\frac{m}{n}\right)\f 
\end{equation}
and hence the map
\[
\f \mapsto \d\f-\imag\left(\frac{m}{n}\right)\psi \f 
\]
equips $K^{m/n}$ with a connection $\nabla$ whose connection form is $-\mathrm{i}(m/n)\psi$. The $(0,\! 1)$-part $\nabla^{\prime\prime}$ of $\nabla$ equips $K^{m/n}$ with a holomorphic line bundle structure $\ov{\partial}_{K^{m/n}}$, so that $\nabla$ is the Chern connection of the Hermitian holomorphic line bundle $(K^{m/n},\ov{\partial}_{K^{m/n}},\hmet_0)$. 

Finally, note that applying $\V$ again to~\eqref{eq:rightactionupinf} shows that we may write $\f=\frac{n\V b}{m}+\imag b$ for a unique real-valued function $b$ on $SM^{1/n}$ satisfying $\V\V b=-\left(\frac{m}{n}\right)^2 b$. Conversely, if a smooth real-valued function $b$ on $SM^{1/n}$ satisfies $\V\V b=-(\frac{m}{n})^2 b$, then $\f:=\frac{n\V b}{m}+\imag b$ represents a smooth section $B$ of $K^{m/n}$.  

\subsection{Thermostats} 

Let $N$ be a smooth $3$-manifold equipped with three smooth vector fields $(X,H,V)$ that are linearly independent at each point of $N$. Following~\cite{Bry} we define:
\begin{Definition}
We say $N$ carries a~\textit{generalised Riemannian structure} if $(X,H,V)$ satisfy the commutator relations
\begin{equation}\label{eq:bracket}
[V,X]=H,\qquad [V,H]=-X, \qquad [X,H]=K_g V,
\end{equation}
for some smooth function $K_g$ on $N$.
\end{Definition} 

\begin{Example}\label{ex:genriemstruc}
Let $(M,g)$ be an oriented Riemannian $2$-manifold and $\pi_n : SM^{1/n} \to M$ an $n$-th root of its unit tangent bundle $\pi : SM \to M$. Then $(X,H,\V)$ defined as in \cref{subsec:roots} equip $N=SM^{1/n}$ with a generalised Riemannian structure. 
\end{Example}

Suppose $N$ carries a generalised Riemannian structure $(X,H,V)$ with dual $1$-forms $(\omega_1,\omega_2,\psi)$. 
\begin{Definition}
A~\textit{(generalised) thermostat} on $N$ is a flow $\phi$ generated by a vector field of the form $\F:=X+\lambda V$, where $\lambda \in C^{\infty}(N)$. 
\end{Definition}

\section{Dominated splittings and hyperbolicity}\label{sec:dom+hyp}

In this section we summarize the main dynamical set up that we shall use; in the first three subsections we follow closely the presentation in \cite{GabrielThomas_Thermo}. For background on the notion of dominated splittings we refer to \cite{CP2015}.

\subsection{Definitions}

Let $N$ be a smooth closed $3$-manifold and $\phi : N\times \R \to N$ a continuous flow. A~\textit{cocycle over $\phi$ with values in $\mathrm{GL}(2,\R)$} is a continuous map $\Psi : N \times \R \to \mathrm{GL}(2,\R)$ such that
\[
\Psi_{t_1+t_2}(x)=\Psi_{t_1}(\phi_{t_2}(x))\Psi_{t_2}(x)
\]
for all $t_1,t_2 \in \R$ and $x \in N$. Note that the cocycle condition ensures that on the trival vector bundle $E=N\times \R^{2}$ we obtain a continuous linear flow $\rho : E \times \R \to E$ by defining
\[
\rho_t((x,a))=(\phi_t(x),\Psi_t(x)a)
\]
for all $(x,a) \in E=N\times \R^2$ and $t \in \R$. 

We say $E$ admits a continuous $\rho$-invariant splitting if there exist continuous $\rho$-invariant line bundles $E^{s,u}$ so that $E=E^u\oplus E^s$. We fix a norm $|\cdot|$ on $\R^2$. 
\begin{Definition} The cocycle $\Psi$ is said to be~\textit{hyperbolic} is there exists a continuous $\rho$-invariant splitting $(E^s,E^u)$ and positive constants $C,\mu>0$ so that
\[
\Vert\left.\Psi_t(x)\right|_{E^s(x)}\Vert \leqslant C\mathrm{e}^{-\mu t}\quad\text{and}\quad \Vert\left.\Psi_{-t}(x)\right|_{E^u(x)}\Vert \leqslant C\mathrm{e}^{-\mu t}
\]
for all $x \in N$ and $t>0$.
\end{Definition} Here $\Vert\cdot\Vert$ denotes the operator norm induced on $\mathrm{Hom}(E^{s,u}(x),E^{s,u}(\phi_t(x)))$ by the norm $|\cdot |$, respectively. A weaker notion than that of hyperbolicity is to ask that for all $x \in N$, any direction not contained in the suspace $E^{s}(x)$ converges exponentially fast to $E^u(\phi_t(x))$ when applying $\rho_t(x)$. This condition is equivalent to the following notion:
\begin{Definition}
The cocycle $\Psi$ is said to admit a~\textit{dominated splitting} if there exists a continuous $\rho$-invariant splitting $(E^u,E^s)$ and positive constants $C,\mu>0$ so that
\begin{equation}\label{eq:domsplitting}
\Vert\left.\Psi_t(x)\right|_{E^s(x)}\Vert \Vert\left.\Psi_{-t}(\phi_t(x))\right|_{E^u(\phi_t(x))}\Vert\leqslant C\mathrm{e}^{-\mu t}
\end{equation}
for all $x \in N$ and $t>0$. 
\end{Definition}

\subsection{The derivative cocycle of a thermostat}
Suppose the closed $3$-manifold $N$ is equipped with a generalised Riemannian structure and a thermostat $\phi$ generated by the vector field $F=X+\lambda V$ as above. Using the bracket relations~\eqref{eq:bracket}, it is straightforward to derive the ODEs dictating the behavior of $d\phi_{t}$. Given an initial condition $\xi\in T_xN$ and if we write
\[d\phi_{t}(\xi)=w(t)F(\phi_t(x))+y(t)H(\phi_t(x))+u(t)V(\phi_t(x))\]
for real-valued functions $w,y,u$ on $\R$, then 
\begin{align}
\dot{w} &=\lambda\,y;\label{jac1}\\
\dot{y} &=u;\label{jac2}\\
\dot{u} &=V(\lambda)\dot{y}-\kappa y,\label{jac3}
\end{align}
where 
\begin{equation}\label{eq:defkappa}
\kappa:=K_g-H\lambda+\lambda^2.
\end{equation}

In order to associate a cocycle to a thermostat we consider the rank two quotient vector bundle $E=TN/\R F\simeq \R H\oplus \R V$. Elements in $E$ will be denoted by $[\xi]$, where $\xi \in TN$. The mapping $d\phi_t$ descends to define a mapping  
$$
\rho : \R\times E \to E, \quad (t,[\xi])\mapsto \rho(t,[\xi])=[d\phi_t(\xi)] 
$$
which satisfies $\rho_{t_1}\circ\rho_{t_2}=\rho_{t_1+t_2}$ for all $t_1,t_2 \in \R$. This is sometimes called the {\it linear Poincar\'e flow}. The basis of vector fields $(F,H,V)$ on $N$ defines a vector bundle isomorphism $TN \simeq N\times\R^3$ and consequently an identification $E\simeq N \times \R^2$. Therefore, we obtain a cocycle $\Psi : N \times \R \to \mathrm{GL}(2,\R)$ over $\phi$ by requiring that for each $t \in \R$ and all $(x,a) \in E$, we have
\[
\rho_t((x,a))=(\phi_t(x),\Psi_t(x)a).
\]
Explicitly,  $\Psi_t$ is the linear map whose action on $\mathbb{R}^2$ is 
$$
\Psi_t(x): \left( \begin{array}{c} y(0) \\ \dot{y}(0) \end{array} \right) \mapsto \left( \begin{array}{c} y(t) \\ \dot{y}(t) \end{array} \right)
$$
with 
\[\ddot{y}(t)-(V\lambda)(\phi_{t}(x))\dot{y}(t) + \kappa(\phi_t(x)) y(t) = 0.\]
Observe that for thermostats the $2$-plane bundle spanned by $H$ and $V$ is in general {\it not} invariant under $d\phi_{t}$.

The cocycle $\Psi_{t}$ is hyperbolic if and only if the thermostat flow $\phi_t$ is Anosov (cf. for instance \cite[Proposition 5.1] {W1}). We will say that $\phi_t$ admits a dominated splitting if $\Psi_{t}$ admits a dominated splitting. This is the natural notion for flows, see \cite[Definition 1]{AR-H}. For the case of flows on 3-manifolds, as it is our case, the existence of a dominated splitting can produce hyperbolicity if additional information on the closed orbits is available. Indeed \cite[Theorem B]{AR-H} implies that if all closed orbits of $\phi$ are hyperbolic saddles, then $N=\Lambda\cup\mathcal T$ where $\Lambda$ is a hyperbolic invariant set and $\mathcal T$ consists of finitely many normally hyperbolic irrational tori.

Flows with dominated splitting are also called {\it projectively Anosov flows}. 
We note that when the flow $\phi$ admits a dominated splitting we may write $TN=\tilde{E}^{s}+\tilde{E}^{u}$, where
$\tilde{E}^{s,u}$ are continuous plane bundles invariant under $d\phi_t$ and whose intersection is $\R F$.
In general they are integrable but unlike the Anosov case, they may not be uniquely integrable. Also note that
the irrational tori in $\mathcal T$ must be tangent to $\tilde{E}^{s}$ or $\tilde{E}^u$ due to the domination condition.
We refer to \cite{MR2766227} and references therein for a classification of these flows when the bundles $\tilde{E}^{s,u}$ are of class $C^2$ (in which case they do determine codimension one foliations of class $C^2$).

\subsection{Infinitesimal generators and conjugate cocycles}\label{subsection:cocycle} For a smooth cocycle $\Psi:N\times\R\to \mathrm{GL}(2,\R)$, we define its infinitesimal generator $\mathbb{B}:N\to \mathfrak{gl}(2,\R)$ as follows
\[\mathbb{B}(x):=-\left.\frac{\d}{\d t}\right|_{t=0}\Psi_{t}(x).\]
The cocycle $\Psi$ can be obtained from $\mathbb{B}$ as the unique solution to
\[\frac{\d}{\d t}\Psi_{t}(x)+\mathbb{B}(\phi_{t}(x))\Psi_{t}(x)=0,\;\;\;\Psi_{0}(x)=\mbox{\rm Id}.\]
In the case of thermostats, it is easy to check that we have
$$
\mathbb{B}=\begin{pmatrix} 0 & -1 \\ \kappa & -V\lambda\end{pmatrix}
$$
where $\kappa=K_g-H\lambda+\lambda^2$.
Given a gauge, that is, a smooth map $\mathcal{P}:N\to \mathrm{GL}(2,\R)$, we obtain a new cocycle by conjugation
\[\tilde{\Psi}_{t}(x)=\mathcal{P}^{-1}(\phi_{t}(x))\Psi_{t}(x)\mathcal{P}(x).\]
It is straightforward to check that the infinitesimal generator $\tilde{\mathbb{B}}$ of $\tilde{\Psi}_{t}$ is related to $\mathbb{B}$ by
\begin{equation}
\tilde{\mathbb{B}}=\mathcal{P}^{-1}\mathbb{B}\mathcal{P}+\mathcal{P}^{-1}F\mathcal{P}.
\label{eq:conjugate}
\end{equation}

Below we shall use gauges of a particular type. Consider a gauge transformation $\mathcal{P}:N\to \mathrm{GL}(2,\R)$ given by
\[\mathcal{P}=\begin{pmatrix} 1 & 0 \\ p & 1\end{pmatrix},\]
where $p$ is a smooth real-valued function on $N$. 
A computation using \eqref{eq:conjugate} shows that the conjugate cocyle $\tilde{\Psi}_{t}$ via $P$ has infinitesimal generator given by
\[\tilde{\mathbb{B}}=\begin{pmatrix} -p & -1 \\ \kappa_{p} & -V\lambda+p\end{pmatrix},\]
where $\kappa_{p}:=\kappa+Fp+p(p-V\lambda)$. Since the cocycles $\Psi_{t}$ and $\tilde{\Psi}_{t}$ are conjugate, they have the same dominated splitting/hyperbolicity properties, but the form of $\tilde{\mathbb B}$ will expose the origins of these properties when $\kappa_{p}<0$ (cf. \cite[Introduction]{W3}). In both cases, the trace of the matrix is $-V\lambda$ (minus divergence of $F$), giving an indication that $F$ may not preserve volume.


\subsection{Conditions ensuring domination and hyperbolicity}\label{subsection:ode}

We have~\cite[Theorem 3.7]{GabrielThomas_Thermo}:

\begin{Theorem} Let $N$ be a closed $3$-manifold that is equipped with a generalised Riemannian structure $(X,H,V)$ and a thermostat flow $\phi$ generated by $F=X+\lambda V$. Suppose there exists a smooth function $p:N\to\mathbb{R}$ such that
$$\kappa_{p}=\kappa+Fp+p(p-V\lambda)<0.$$ Then 
 $\phi$ admits a dominated splitting with $V\notin E^{s,u}$. 
 \label{thm:workhorse}
\end{Theorem}
\begin{Remark}
More precisely, in~\cite[Theorem 3.7]{GabrielThomas_Thermo}, only the case of a thermostat on the unit tangent bundle of an oriented Riemannian $2$-manifold $(M,g)$ is considered. However, it is easy to check that the arguments in~\cite[Theorem 3.7]{GabrielThomas_Thermo} also prove~\cref{thm:workhorse}. 
In \cite{GabrielThomas_Thermo} we employed quadratic forms to establish this result; we could have used instead a cone-field criterion as described for instance in \cite[Theorem 2.6]{CP2015}. 
\end{Remark}

The fact that $V\notin E^{s,u}$ implies that there are uniquely defined continuous (H\"older in fact) functions $r^{s,u}:N\to \mathbb{R}$ such that $H+r^{s,u} V\in E^{s,u}$. The invariance of the bundles $E^{s,u}$ translates into Riccati equations for $r^{s,u}$ of the form:
\[Fr+r^{2}-rV\lambda+\kappa=0.\]
Observe that $h:=r-p$ satisfies the Riccati equation
\begin{equation}
Fh+h^{2}+h(2p-V\lambda)+\kappa_{p}=0.
\label{eq:ric2}
\end{equation}
Moreover, the functions $r^{u,s}$ can be constructed using a limiting procedure as follows.
Fix $x\in N$ and consider for each $R>0$, the unique solution $u_{R}$ to the Riccati equation
along $\phi_{t}(x)$
\[\dot{u}+u^2-uV\lambda+\kappa=0\]
satisfying $u_{R}(-R)=\infty$. Then
\begin{equation}
r^{u}(x)=\lim_{R\to \infty}u_{R}(0).
\label{eq:hopflimit}
\end{equation}
Note that $r^{u}(\phi_{t}(x))=\lim_{R\to\infty}u_{R}(t)$. 

Finally, under the assumption in \cref{thm:workhorse} that $\kappa_p<0$ we get the important additional information that $h^{u}:=r^{u}-p>0$ and $h^{s}:=r^{s}-p<0$. We call these the positive and negative Hopf solutions given that they play a similar role as the solutions introduced by E. Hopf in \cite{Hopf48} for the geodesic flow.

The property $V\notin E^{s,u}$ allows a convenient visualization of the domination condition in terms of the behaviour of solutions to the Riccati equation as depicted in \cref{Figure:riccati}. 
\begin{figure}[h]
\begin{overpic}[scale=0.45]{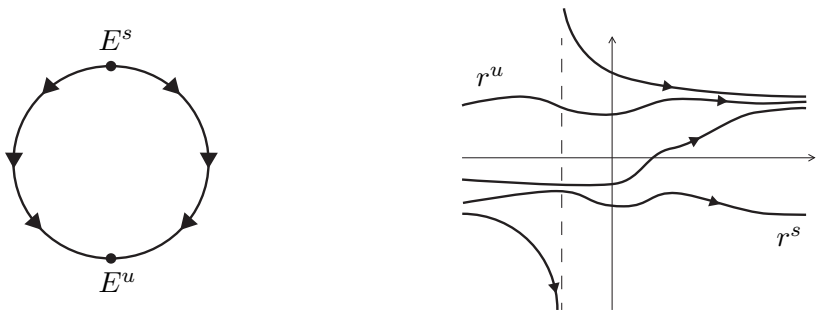}
\put(11.5,32.5){$E^s$}
\put(11.5,2.25){$E^u$}
\put(58,28){$r^u$}
\put(95,8){$r^s$}
\end{overpic}
\caption{Dominated splitting property}\label{Figure:riccati}
\end{figure}
The reader might find this figure useful when following some of the arguments below, particularly the proof of \cref{lemma:fullbound}. To prove that our flows are Anosov we shall use the following lemma that ``upgrades" the domination condition to hyperbolicity under additional information on the solutions $r^{s,u}$.

\begin{Lemma} Under the same assumptions as in \cref{thm:workhorse}, suppose in addition
that either
\begin{enumerate}
\item $r^{u}>0$ and $r^{s}<0$;
or
\item $V\lambda-p-\frac{\kappa_{p}}{r^{u}-p}>0$ and $V\lambda-p-\frac{\kappa_{p}}{r^{s}-p}<0$.
\end{enumerate}
Then $\phi_t$ is Anosov.
\label{lemma:alternative}
\end{Lemma}

\begin{proof} We first consider (1). For a given initial condition $(y(0),\dot{y}(0))\in E^{u}$ we know that
under the coycle $\Phi_t$, we have $\dot{y}=r^{u}y$. If $r_u>0$ we can find a uniform constant
$\mu>0$ such that $|y(-t)|\leqslant  e^{-\mu t}|y(0)|$ for $t>0$. This gives uniform exponential growth for $\Psi_t$
on $E^u$. Arguing with $r^{s}<0$ we get uniform exponential contraction for $\Psi_t$ on $E^{s}$ thus showing that $\Psi_t$ is hyperbolic.

Assume now condition (2) and consider a solution with initial conditions $(y(0),\dot{y}(0))\in E^{u}$. 
Then $\dot{y}=r^{u}y$ and let $z:=(r^{u}-p)y$ (recall that $r^{u}-p>0$). Then a calculation
shows that $\dot{z}=(V\lambda-p)z-\kappa_{p}y=(V\lambda-p-\frac{\kappa_{p}}{r^{u}-p})z$.
This gives exponential growth for $z$ and hence the desired exponential
growth for $\Psi_{t}$ on $E^{u}$. Arguing in a similar way with $E^{s}$, we deduce that $\Psi_{t}$ is hyperbolic.
\end{proof}

\begin{Remark} In \cite{GabrielThomas_Thermo} we used condition (1) to prove that thermostat flows with $\theta=0$ are Anosov when $\ell$ is an integer $\geq 1$.
Remarkably, for the case of fractional differentials in the range $0<\ell<1$, we will crucially need alternative (2).

\label{rem:upgrade}
\end{Remark}

While we shall not use the next proposition, it complements \cref{thm:workhorse} quite nicely
and it gives an indication of the importance of the property $V\notin E^{s,u}$.

\begin{Proposition} Suppose the thermostat determined by $\lambda$ is such that $\Psi_t$ admits a continuous invariant splitting $E=E^u\oplus E^s$ with $V\notin E^{u,s}$. Then the splitting is dominated and there
exists a hyperbolic $\mathrm{SL}(2,\R)$-cocycle $\Psi^{hyp}_{t}$ such that
\[\Psi_{t}=e^{\frac{1}{2}\int_{0}^{t}V\lambda}\;\Psi_{t}^{hyp}.\]
\label{prop:cs}
\end{Proposition}

\begin{proof} We know that the existence of a splitting with $V\notin E^{u,s}$ gives rise to two continuous functions $r^{u,s}:N\to\R$ satisfying the Riccati equation
\[Fr+r^2-rV\lambda+\kappa=0.\]
Moreover, $r^u-r^s\neq 0$.

Recall that the infinitesimal generator for the cocycle $\Psi_t$ is:
$$
\mathbb{B}=\begin{pmatrix} 0 & -1 \\ \kappa & -V\lambda\end{pmatrix}.
$$
Consider a gauge transformation $\mathcal{P}:N\to \mathrm{GL}(2,\R)$ given by
\[\mathcal{P}=\begin{pmatrix} 1 & 0 \\ p & 1\end{pmatrix}\]
with $p=\frac{V\lambda}{2}$. Then the conjugate cocyle $\tilde{\Psi}_{t}$ via $\mathcal{P}$ has infinitesimal generator given by
\[\tilde{\mathbb{B}}=-\frac{1}{2}V\lambda\;\text{\rm Id}+\begin{pmatrix} 0 & -1 \\ \kappa_{p} & 0\end{pmatrix}.\]
To complete the proof we need to prove that the cocycle generated by
\[\begin{pmatrix} 0 & -1 \\ \kappa_{p} & 0\end{pmatrix}\]
is hyperbolic. Note that $h^{u,s}:=r^{u,s}-p$ satisfies the Riccati equation
\[Fh+h^{2}+\kappa_{p}=0.\]
The quadratic form
\[Q(a,b)=2ab-([h^{u}]^{2}+[h^{s}]^{2})a^2\]
has the property that
\[\dot{Q}=(b-h^{u}a)^{2}+(b-h^{s}a)^{2}>0\]
unless $a=b=0$. (Note that $\dot{b}+\kappa_{p}a=0$ and $\dot{a}=b$.) Now the hyperbolicity follows for instance from ~\cite[Proposition 4.1 \& Theorem 4.4]{W3}.
\end{proof}

\begin{Remark} We do not know of any example of a thermostat as in \cref{prop:cs} that is not
Anosov.
\end{Remark}

\subsection{Bi-contact structures} It is possible to recast the discussion of Subsection \ref{subsection:ode} in terms of the notion of bi-contact structure introduced by Eliashberg and Thurston \cite{EM98} and further studied by Mitsumatsu \cite{Mit95}
in the context of projective Anosov flows. 

If $N$ is a closed 3-manifold, we shall say that a {\it bi-contact pair} is a pair of contact forms $(\tau_{+},\tau_{-})$
such that $\tau_{+}\wedge d\tau_{+}$ and $\tau_{-}\wedge d\tau_{-}$ give rise to opposite orientations and
$\text{ker}\,\tau_{+}\cap \text{ker}\,\tau_{-}$ is 1-dimensional at every point.
It turns out (cf. \cite{EM98,Mit95}) that the flow of a non-zero vector field $F$ has a dominated splitting (or is a projective Anosov flow) iff there is a bi-contact pair $(\tau_{+},\tau_{-})$ such that $F\in \text{ker}\,\tau_{+}\cap \text{ker}\,\tau_{-}$.

Suppose now that $N$ is endowed with a generalized Riemannian structure $(X,H,V)$ and $\lambda,p\in C^{\infty}(N)$ are given functions. We consider a new frame $(F,H_{p},V)$, where $F:=X+\lambda V$ and
$H_{p}=H+pV$. If we denote by $(\alpha,\beta,\psi)$ the co-frame dual to $(X,H,V)$, then a simple computation
shows that $(\alpha,\beta,\tilde{\psi})$ is the co-frame dual to $(F,H_{p},V)$, where
\[\tilde{\psi}=-\lambda\alpha-p\beta+\psi.\] 
Then we have:

\begin{Lemma} The pair $(\beta,\tilde{\psi})$ is a bi-contact pair iff $\kappa_{p}<0$.
\end{Lemma}

We omit the proof of the lemma (which is a fairly straightforward computation), since we will not use it in subsequent sections. Since $F\in \text{ker}\,\beta\cap \text{ker}\,\tilde{\psi}$ we see that with this lemma
we essentially recover \cref{thm:workhorse}.
The conditions appearing in \cref{lemma:alternative} can now be rephrased in a more pleasing way
in terms of the bi-contact pair $(\beta,\tilde{\psi})$. Indeed condition (1) is equivalent to
\[d\beta(F,H+r^{u}V)>0\;\;\text{and}\;\;d\beta(F,H+r^{s}V)<0\]
while condition (2) is equivalent to
\[d\tilde{\psi}(F,H+r^{u}V)>0\;\;\text{and}\;\;d\tilde{\psi}(F,H+r^{s}V)<0.\]
Again, we omit the verification of these equivalences as they will not be used in the sequel.

\section{Thermostats from Vortices}\label{sec:vortextothermo}

\subsection{The vortex equations}

Let $(M,g)$ be a closed oriented Riemannian $2$-manifold of negative Euler characteristic and $\nu : L \to M$ a complex line bundle of positive degree. For a triple consisting of a Hermitian bundle metric $\hmet$ on $L$, a del-bar operator $\ov{\partial}_L$ on $L$, and a $(1,\! 0)$-form $\varphi$ on $M$ with values in $L$, we consider the following pair of equations
\[
R(\D)+\frac{1}{2}\varphi\wedge\varphi^*+\imag \ell\Omega_g=0 \qquad \text{and} \qquad \ov{\partial}_L\varphi=0. 
\]
Here we write $\ell:=\deg(L)/|\chi(M)|$, $\D$ denotes the Chern connection on $L$ with respect to $(\hmet,\ov{\partial}_L)$, $R(\D)$ its curvature, $\Omega_g$ the area form of $g$ and the $1$-form $\varphi^*$ with values in the dual $L^{-1}$ of $L$ is defined by 
\[
\varphi^*(v)(\xi):=\hmet(\xi,\varphi(v))
\] for all $x \in M$, $v \in TM$ and $\xi \in \nu^{-1}(\{x\})$. We assume $\hmet$ to be conjugate linear in the second variable, so that $\varphi^*$ is an $L^{-1}$-valued $(0,\! 1)$-form. We extend the wedge-product to bundle-valued forms in the standard way, so that for $\varphi \in \Omega^{1}(L)$ and $\varrho \in \Omega^1(L^{-1})$, we have
\[
(\varphi\wedge\varrho)(v,w)=\varrho(w)\varphi(v)-\varrho(v)\varphi(w)
\] 
for all $x \in M$ and $v,w \in T_xM$. In particular, we obtain
\begin{align*}
\left(\varphi\wedge\varphi^*\right)(v,w)&=\hmet(\varphi(v),\varphi(w))-\hmet(\varphi(w),\varphi(v))\\
&=\hmet(\varphi(v),\varphi(w))-\ov{\hmet(\varphi(v),\varphi(w))}=2\imag \Im \hmet(\varphi(v),\varphi(w))
\end{align*}
so that $\varphi\wedge\varphi^*$ is a purely imaginary $(1,\! 1)$-form on $M$. 

The~\textit{complex gauge group $\mathrm{G}_{\C}$} of $L$ is the group of automorphisms of $L$ (covering the identity on $M$) and the~\textit{gauge group $\mathrm{G}$} of $(L,\hmet)$ consists of the automorphisms of $L$ that are unitary with respect to $\hmet$. Since an automorphism of a one-dimensional complex vector space is just a non-vanishing complex number, we have $\mathrm{G}_{\C}\simeq C^{\infty}(M,\C^*)$ and $\mathrm{G}\simeq C^{\infty}(M,\mathrm{U}(1))$, the smooth functions on $M$ with values in the one-dimensional unitary group $\mathrm{U}(1)$. An element $\tau \in \mathrm{G}_{\C}$ acts on a Hermitian bundle metric $\hmet$ on $L$ by the rule
\begin{equation}\label{eq:gaugeactionmet}
\tau\cdot \hmet=|\tau|^2\hmet
\end{equation}
and on $\varphi \in \Omega^{p,q}(L)$ by the rule
\begin{equation}\label{eq:gaugeactionhiggs}
\tau\cdot \varphi=\tau^{-1}\varphi.
\end{equation}
We define an action on the space of del-bar operators on $L$ by 
\begin{equation}\label{eq:gaugeactiondelbar}
\tau\cdot \ov{\partial}_L=\ov{\partial}_L+\tau^{-1}\ov{\partial}\tau.  
\end{equation}
Writing $\D_{\hmet,\ov{\partial}_L}$ for the Chern connection on $L$ determined by the Hermitian metric $\hmet$ and del-bar operator $\ov{\partial}_L$, we obtain:
\begin{Lemma}\label{lem:curvchangechern}
For a Hermitian holomorphic line bundle $(L,\hmet,\ov{\partial}_L)$ and $\tau \in \mathrm{G}_{\C}$ we have the following identities:
\begin{itemize}
	\item[(i)]
$R(\D_{\tau \cdot \hmet,\ov{\partial}_L})=R(\D_{\hmet,\ov{\partial}_L})-2\partial\ov{\partial}\log|\tau|$,
\item[(ii)] $R(\D_{\hmet,\tau\cdot\ov{\partial}_L})=R(\D_{\hmet,\ov{\partial}_L})+2\partial\ov{\partial}\log|\tau|$. 
\end{itemize}
\end{Lemma}
\begin{proof}
(i) : Let $s : U \to L$ be a local non-vanishing holomorphic section of $L$. We write $u:=\hmet(s,s)$ and let $\theta \in \Omega^1_U$ denote the connection form of the Chern connection $\D_{\hmet,\ov{\partial}_L}$ with respect to $s$. Recall that $\theta=u^{-1}\partial u$. Therefore, the connection form $\theta^{\prime}$ of the Chern connection $\D_{\tau\cdot \hmet,\ov{\partial}_L}$ with respect to $s$ satisfies
\[
\theta^{\prime}=(|\tau|^2u)^{-1}\partial(|\tau|^2u)=\theta+2\partial \log|\tau|
\]
The curvature thus becomes
\[
\d \theta^{\prime}=\d\theta -2\partial\ov{\partial} \log|\tau|
\]
which proves (i). In order to prove (ii) we first remark that the connection
\[
\D=\D_{\hmet,\ov{\partial}_L}+\tau^{-1}\ov{\partial}\tau
\]
satisfies $\D^{\prime\prime}=\D^{\prime\prime}_{\hmet,\tau\cdot\ov{\partial}_L}$ and thus so does 
\[
\nabla=\D_{\hmet,\ov{\partial}_L}+\tau^{-1}\ov{\partial}\tau-\ov{\tau}^{-1}\partial\ov{\tau}
\]
as we have added a $(1,\! 0)$-form. By definition the Chern connection $\D_{\hmet,\ov{\partial}_L}$ is compatible with $\hmet$ and hence so is $\nabla$, as we have added a purely imaginary $1$-form. Therefore $\nabla$ is compatible with $\hmet$ and satisfies $\nabla^{\prime\prime}=\D^{\prime\prime}_{\hmet,\tau\cdot\ov{\partial}_L}$, so it must be the Chern connection $\D_{\hmet,\tau\cdot\ov{\partial}_L}$. For the curvature we obtain
\[
R(\D_{\hmet,\tau\cdot\ov{\partial}_L})=R(\D_{\hmet,\ov{\partial}_L})+\d\left(\tau^{-1}\ov{\partial}\tau-\ov{\tau}^{-1}\partial\ov{\tau}\right)=R(\D_{\hmet,\ov{\partial}_L})+2\partial\ov{\partial}\log |\tau|.\qedhere
\]
\end{proof}
We now have:
\begin{Proposition}
Let $L \to M$ be a complex line bundle on the oriented Riemannian $2$-manifold $(M,g)$ and $\ell:=\deg(L)/|\chi(M)|$. Then the triple $(\hmet,\ov{\partial}_L,\varphi)$ satisfies
\[
R(\D)+\frac{1}{2}\varphi\wedge\varphi^*+\imag \ell\Omega_g=0 \quad \text{and} \quad \ov{\partial}_L\varphi=0 
\]
if and only if $(\tau\cdot \hmet,\tau\cdot\ov{\partial}_L,\tau\cdot\varphi)$ does. 
\end{Proposition}
\begin{proof}
We observe that for all $v,w \in TM$
\begin{multline*}
\left((\tau\cdot\varphi)\wedge(\tau\cdot\varphi)^{*_{\tau\cdot \hmet}}\right)(v,w)=|\tau|^2\hmet(\tau^{-1}\varphi(w),\tau^{-1}\varphi(v))\\-|\tau|^2\hmet(\tau^{-1}\varphi(v),\tau^{-1}\varphi(w))=|\tau|^2\tau^{-1}\ov{\tau^{-1}}(\varphi\wedge\varphi^{*_\hmet})(v,w)=(\varphi\wedge\varphi^{*_\hmet})(v,w)
\end{multline*}
so that $\varphi\wedge\varphi^* \in \Omega^{1,1}$ is invariant under complex gauge transformations. Now \cref{lem:curvchangechern} immediately implies that
$R(\D_{\hmet,\ov{\partial}_L})=R(\D_{\tau\cdot \hmet,\tau\cdot\ov{\partial}_L})$
thus showing the invariance of the first equation. Likewise, we immediately obtain
\[
(\tau\cdot \ov{\partial}_L)(\tau\cdot \varphi)=\tau\cdot\ov{\partial}_L\varphi,
\]
so that the equation 
\[
\ov{\partial}_L\varphi=0 
\]
is preserved under the action of the complex gauge group.  
\end{proof}

\subsection{The vortex equations on a root of $SM$}

Since $L$ has positive degree and $\chi(M)<0$, there exist unique positive coprime integers $(m,n)$ so that we have an isomorphism $L^n\simeq K^m$ of complex line bundles. We fix an $n$-th root $SM^{1/n}$ of the unit tangent bundle $SM$ of $(M,g)$ and let $K^{1/n}$ denote the corresponding $n$-th root of $K$, so that we have an isomorphism $\mathcal{Z} : L\to K^{m/n}$ of complex line bundles. Note that such a root exists since $n$ divides $\chi(M)$. We equip $SM^{1/n}$ with the generalised Riemannian structure $(X,H,\mathbb{V})$ as in~\cref{ex:genriemstruc}. We may write $\hmet=\mathrm{e}^{2f}\hmet_0$ for a unique smooth real-valued function $f$ on $M$. Abusing notation, we also use the letter $f$ to denote the pullback of $f$ to $SM^{1/n}$. Recall that the space of del-bar operators on a line bundle $L \to M$ is an affine space modelled on $\Omega^{0,1}$. Therefore, without loosing generality, we can assume that there exists a $1$-form $\theta$ on $M$ so that 
\begin{equation}\label{eq:difcplxstruc}
\ov{\partial}_L=\ov{\partial}_{K^{m/n}}-\ell\,\theta^{0,1},
\end{equation}
where $\theta^{0,1}=\frac{1}{2}(\theta-\imag\star_g\theta) \in \Omega^{0,1}$ denotes the $(0,\! 1)$-part of $\theta$ and $\star_g$ the Hodge-star with respect to $g$. We may also think of $\theta$ as a real-valued function on $SM$ and abusing notation, we also write $\theta$ to denote its pullback to $SM^{1/n}$. Note that the function $\theta$ on $SM^{1/n}$ satisfies $\V\V\theta=-\theta$. The pullback of $\theta^{0,1}$ to $SM^{1/n}$ can be expressed as $\frac{1}{2}(\theta+\imag \V\theta)\ov{\omega}$, where we write $\omega=\omega_1+\imag\omega_2$ and $\ov{\omega}=\omega_1-\imag\omega_2$. Therefore, the connection form $\zeta$ on $SM^{1/n}$ of the Chern connection $\D$ of $(L,\ov{\partial}_L,\hmet)$ can be written as
\[
\zeta=-\imag\ell\psi+w\omega-\frac{\ell}{2}(\theta+\imag \V\theta)\ov{\omega}
\]
for some unique complex-valued function $w$ on $SM^{1/n}$. On $SM^{1/n}$, the condition that $\D$ preserves $\hmet=\mathrm{e}^{2f}\hmet_0$ translates to
\[
\d\left(\mathrm{e}^{2f}\f_1\ov{\f_2}\right)=\mathrm{e}^{2f}\Big((\d \f_1+\zeta\f_1)\ov{\f_2}+\f_1(\d \ov{\f_2}+\ov{\zeta}\ov{\f_2})\Big)
\]
where $\f_1,\f_2$ represent arbitrary smooth sections of $L$. A straighforward calculation yields
\[
\zeta=-\imag\ell\psi+\left(\frac{\ell}{2}(\theta-\imag \V\theta)+Xf-\imag Hf\right)\omega-\frac{\ell}{2}(\theta+\imag \V\theta)\ov{\omega}. 
\]
The $(1,\! 0)$-form $\varphi$ with values in $L$ is a section of $K\otimes L\simeq K^{(n+m)/n}$, so that on $SM^{1/n}$ the form $\varphi$ is represented by a complex-valued $1$-form $\bm{\varphi}$, which we may write as
\[
\bm{\varphi}=\ell\left(\frac{\V a}{1+\ell}+\imag a\right)\omega,
\]
where the real-valued function $a$ satisfies $\V\V a=-(1+\ell)^2 a$, since $\ell=m/n$. 
\begin{Lemma}\label{lem:holcond}
We have $\ov{\partial}_L\varphi=0$ if and only if
\begin{equation}\label{eq:vortexII}
0=X\V a-(1+\ell)Ha-\ell\theta \V a+\ell(1+\ell)a\V\theta.
\end{equation}
\end{Lemma}
\begin{proof}
Since $M$ is complex one-dimensional, the condition $\ov{\partial}_L\varphi=0$ is equivalent to $\varphi$ being covariant constant with respect to the Chern connection $\D$ of $(L,\hmet,\ov{\partial}_L)$. On $SM^{1/n}$ this translates to
\[
0=\d\bm{\varphi}+\zeta\wedge\bm{\varphi}. 
\]
Since $\zeta$ defines a connection on $L$, terms involving $\psi$ will cancel each other out and hence we can compute modulo $\psi$. We obtain
\[
\zeta\wedge\bm{\varphi}=\frac{\ell^2}{2}\left(\frac{\V a}{(1+\ell)}+\imag a\right)(\theta+\imag \V\theta)\,\omega\wedge\ov{\omega}\quad \text{mod}\quad \psi
\]
We define
\[
W_{\pm}=\frac{1}{2}\left(X\mp \imag H\right).
\]
Note that $(W_{+},W_{-},\V)$ is the dual basis to $(\omega,\ov{\omega},\psi)$. Hence we obtain
\[
\d \bm{\varphi}=\ell W_{-}\left(\frac{\V a}{1+\ell}+\imag a\right)\ov{\omega}\wedge\omega=-\frac{\ell}{2} \left(X+\imag H\right)\left(\frac{\V a}{1+\ell}+\imag a\right)\omega\wedge\ov{\omega} \quad \text{mod} \quad \psi
\]
The vanishing of the imaginary part of $\d \bm{\varphi}+\zeta\wedge\bm{\varphi}$ is thus equivalent to
\begin{align*}
0&=\frac{\ell}{1+\ell}X\V a-\ell H a-\frac{\ell^2}{1+\ell}\theta \V a+\ell^2 a \V\theta\\
&=\frac{\ell}{1+\ell}\Big(X\V a-(1+\ell)H a-\ell \theta \V a+\ell(1+\ell) a \V\theta\Big),
\end{align*}
as claimed. 

Conversely, if $a,\theta$ satisfy~\eqref{eq:vortexII}, then applying $\V$ and using the commutator relations~\eqref{eq:comrelation} as well as $\V\V a=-(1+\ell)^2 a$ and $\V\V\theta=-\theta$ easily recovers that the real part of $\d\bm{\varphi}+\zeta\wedge\bm{\varphi}$ must vanish as well. 
\end{proof}
Writing
\[
\bm{A}:=\frac{\V a}{1+\ell}+\imag a,
\]
we obtain:
\begin{Lemma}\label{lem:curvcond}
We have $R(\D)+\frac{1}{2}\varphi\wedge\varphi^*+\imag \ell\Omega_g=0$ if and only if
\begin{equation}\label{eq:vortexI}
K_g+X\theta+H\V\theta=-1+\ell \mathrm{e}^{2f}|\bm{A}|^2-\frac{1}{\ell}\left(XXf+HHf\right)
\end{equation}
\end{Lemma}
\begin{proof}
Observe that $\varphi^* \in \Omega^{0,1}(L^{-1})$ is represented by 
\[
\bm{\varphi}^{\bm{*}}=\mathrm{e}^{2f}\ov{\bm{\varphi}}=\mathrm{e}^{2f}\ell\ov{\bm{A}}\ov{\omega}
\]
so that $\varphi\wedge\varphi^*$ is represented by
\[
\bm{\varphi}\wedge\bm{\varphi}^{\bm{*}}=\ell^2\mathrm{e}^{2f}|\bm{A}|^2\omega\wedge\ov{\omega}. 
\]
Note that the pullback to $SM^{1/n}$ of the area form $\Omega_g$ of $g$ becomes $\frac{\mathrm{i}}{2}\omega\wedge\ov{\omega}$. Again, since $\zeta$ is the connection form of a connection, the $\psi$-terms will cancel each other out in the curvature expression $\d\zeta$. We obtain
\begin{align*}
\d \zeta&=-\frac{\ell}{2}\left(K_g+W_{-}(\theta-\imag\V\theta)+W_{+}(\theta+\imag \V\theta)+\frac{4}{\ell}W_{-}W_{+}f\right)\omega\wedge\ov{\omega}\\
&=-\frac{\ell}{2}\left(K_g+X\theta+H\V\theta+\frac{1}{\ell}(XXf+HHf)\right)\omega\wedge\ov{\omega}, 
\end{align*}
where we use that $Xf-\imag Hf=2W_{+}f$ and the structure equation
\[
\d\psi=-\frac{\i}{2}K_g\omega\wedge\ov{\omega}. 
\]
In total, we get
\begin{multline*}
\d\zeta+\frac{1}{2}\bm{\varphi}\wedge\bm{\varphi}^*+\imag \ell \frac{\i}{2}\omega\wedge\ov{\omega}=-\frac{\ell}{2}\bigg(K_g+X\theta+H\V\theta+\frac{1}{\ell}(XXf+HHf)\bigg.\\
\bigg.-\ell \mathrm{e}^{2f}|\bm{A}|^2+1\bigg)\omega\wedge\ov{\omega}=0,
\end{multline*}
which proves the claim. 
\end{proof}

\subsection{Fractional differentials}
Note that we may think of $\varphi/\ell$ as a section of $K\otimes L\simeq K^{(n+m)/n}$ which we denote by $A$. Thus, we may interpret $A$ as a differential of fractional degree $(n+m)/n=1+\ell$. Recall that the choice of an $n$-th root $SM^{1/n}$ of $SM$ equips $K^{(n+m)/n}$ with a Hermitian bundle metric which we denote by $\hmet_0$. Defining $|A|^2_g:=\hmet_0(A,A)$, the pullback of the function $|A|^2_g$ to $SM^{1/n}$ is $|\bm{A}|^2$. Moreover, the co-differential $\delta_g\theta$ of $\theta$ with respect to $g$ pulls-back to $SM^{1/n}$ to become $-X\theta-H\V\theta$ and the Laplacian $\Delta_g f$ of $f$ with respect to $g$ pulls-back to $SM^{1/n}$ to become $XXf + HHf$. Using this notation, the equation~\eqref{eq:vortexI} can be written as
\[
K_g-\delta_g\theta=-1+\ell \mathrm{e}^{2f}|A|^2_g-\frac{1}{\ell}\Delta_g f. 
\]
Observe also that since $\ov{\partial}_L\varphi=0$, the equation~\eqref{eq:difcplxstruc} implies 
\[
\ov{\partial}_{K^{1+\ell}}A=\ell\,\theta^{0,1}\otimes A. 
\]
\subsection{The thermostat}

In order to associate a thermostat on $SM^{1/n}$ to a solution of the vortex equation, we first consider as a motivating example the case $L=K^2$. In this case $n=1$ and $m=2$ so that no choice of a root of $SM$ is necessary. We may take $\ov{\partial}_L$ to be the del-bar operator on $K^2$ induced by the metric $g$, that is, we choose $\theta$ to vanish identically. Furthermore we choose $\hmet$ to be $\hmet_0$ so that $f$ vanishes identically as well. Thinking of $\varphi$ as a section of $K\otimes L\simeq K^3$, we obtain a cubic differential $A$, and the vortex equations become
\[
K_g=-1+2|A|^2_g\qquad\text{and}\qquad \ov{\partial}_{K^3} A=0. 
\]
In particular, the cubic differential $A$ is holomorphic with respect to the standard holomorphic line bundle structure on $K^3$. Now observe that $L$ admits a square root $L^{1/2}\simeq K$ and hence we may interpret $\varphi/2$ as a section of $K\otimes \mathrm{Hom}(L^{-1/2},L^{1/2})$. Using the Hermitian metric induced by $\hmet_0$ on $L^{1/2}\simeq K$, we may identify $L^{1/2}\simeq \ov{L^{-1/2}}$. As a real vector bundle $\ov{L^{-1/2}}$ is isomorphic to $L^{-1/2}$. Therefore, we may interpret $\varphi/2$ as a $1$-form on $M$ with values in the endomorphisms of $L^{-1/2}$, thought of as a real vector bundle. Identifying $\C\simeq \R^2$ in the usual way, multiplication with the complex number $z$, thought of as a linear map $\R^2 \to \R^2$, has matrix representation
\[
\begin{pmatrix} \Re z & - \Im z \\ \Im z & \Re z \end{pmatrix}
\]
with respect to the standard basis of $\R^2$. Taking into account the identification $L^{1/2}\simeq \ov{L^{-1/2}}$, which just amounts to complex conjugation, the $1$-form $\varphi/2$ is thus represented by
\[
\frac{1}{2}\begin{pmatrix} 1 & 0 \\ 0 & -1 \end{pmatrix}\begin{pmatrix} \Re\bm{\varphi} & -\Im \bm{\varphi} \\ \Im\bm{\varphi} & \Re\bm{\varphi}\end{pmatrix}=\frac{1}{2}\begin{pmatrix} \Re\bm{\varphi} & -\Im\bm{\varphi}\\  -\Im\bm{\varphi} & -\Re\bm{\varphi}\end{pmatrix}.
\]
The Chern connection on $L$ induces a connection on $L^{-1/2}$ whose connection form is $-(1/2)\zeta$. Adding $\varphi/2$ to this connection, thought of as a connection on the real vector bundle $L^{-1/2}$, we obtain a connection $\nabla$ with connection form
\[
\Upsilon=(\Upsilon^i_j)=-\frac{1}{2}\begin{pmatrix} \Re(\zeta-\bm{\varphi}) & -\Im(\zeta-\bm{\varphi}) \\ \Im(\zeta+\bm{\varphi}) & \Re(\zeta+\bm{\varphi}) \end{pmatrix} 
\]
Since $L^{-1/2}\simeq K^{-1}$, the vector bundle $L^{-1/2}$, as a real vector bundle, is isomorphic to the tangent bundle of $M$. Thus $\Upsilon$ defines a connection $\nabla$ on $TM$ and in~\cite[Lemma 3.1]{MR4109900} it is shown that the orbits of the thermostat $\phi$ on $SM$ defined by the condition $F \inc \Upsilon^2_1=0$ project to $M$ to become the geodesics of $\nabla$, when ignoring the parametrisation. 
\begin{Remark}
The connection $\nabla$ defines a properly convex projective structure on $M$ whose associated Hilbert geodesic flow is a $C^1$ reparametrisation of $\phi$. We refer the reader to~\cite{GabrielThomas_Thermo} and references therein for details. 
\end{Remark}

In general $L$ will not admit a square root, but we may nonetheless formally carry out the same construction, except that now the identification $L^{1/2}\simeq \ov{L^{-1/2}}$ needs to amount for the metric $e^f\hmet_0$ induced by $\hmet$ on the formal root $L^{1/2}$. We may thus define
\begin{align}
\begin{split}\label{eq:formcon}
\Upsilon=(\Upsilon^i_j)&=-\frac{1}{2}\begin{pmatrix} \Re\zeta & -\Im \zeta \\ \Im\zeta & \Re\zeta\end{pmatrix}+\frac{1}{2}\begin{pmatrix} \mathrm{e}^f & 0 \\ 0 & -\mathrm{e}^f\end{pmatrix}\begin{pmatrix} \Re\bm{\varphi} & -\Im \bm{\varphi} \\ \Im\bm{\varphi} & \Re\bm{\varphi}\end{pmatrix}\\
&=-\frac{1}{2}\begin{pmatrix} \Re(\zeta-\mathrm{e}^f\bm{\varphi}) & -\Im(\zeta-\mathrm{e}^f\bm{\varphi}) \\ \Im(\zeta+\mathrm{e}^f\bm{\varphi}) & \Re(\zeta+\mathrm{e}^f\bm{\varphi}) \end{pmatrix}.
\end{split}
\end{align}
Note that the vortex equations can be written as
\begin{equation}\label{eq:vortexformnotation}
\d \zeta=\frac{\ell}{2}\omega\wedge\ov{\omega}-\frac{1}{2}\mathrm{e}^{2f}\bm\varphi\wedge\ov{\bm{\varphi}}\qquad\text{and}\qquad \d\bm{\varphi}=-\zeta\wedge\bm{\varphi}.
\end{equation}
We also obtain
\[
\d\omega=\left(\zeta/\ell+\frac{1}{2}(\theta+\imag\V\theta)\ov{\omega}\right)\wedge\omega.
\]
From~\eqref{eq:vortexformnotation} we easily conclude
\[
\d\Upsilon+\Upsilon\wedge\Upsilon=\frac{\i}{4}\begin{pmatrix} 0 & -\ell \\ \ell & 0 \end{pmatrix}\omega\wedge\ov{\omega}. 
\]
Again in formal analogy to the case $L=K^2$, we obtain a thermostat $\phi$ on $SM^{1/n}$ by requiring that $F \inc \Upsilon^2_1=0$. Using the notation above, we have
\[
\lambda=\mathrm{e}^f a-\V\theta-\frac{1}{\ell}Hf. 
\]

\begin{Remark}[Gauge invariance]
Recall that the vortex equations are invariant under the action of the complex gauge group $\mathrm{G}_{\C}$. It is thus natural to ask how the gauge group affects the associated thermostat. Choosing $\tau=\mathrm{e}^{w}$ for some smooth real-valued function $w$ on $M$, the equations \eqref{eq:gaugeactionmet}, \eqref{eq:gaugeactionhiggs} and \eqref{eq:gaugeactiondelbar} imply that the triple $(A,\theta,f)$ is replaced by
\[
(A,\theta,f) \mapsto (\hat{A},\hat{\theta},\hat{f})=(\mathrm{e}^{-w}A,\theta-\frac{1}{\ell}\d w,f+w).
\]
Let $\hat{\lambda}$ be defined with respect to $(\hat{A},\hat{\theta},\hat{f})$. Then we obtain
\[
\hat{\lambda}=\mathrm{e}^{\hat{f}}\hat{a}-\V\hat{\theta}-\frac{1}{\ell}{H\hat{f}}=\mathrm{e}^{f+w}\mathrm{e}^{-w}a-\V\left(\theta-\frac{1}{\ell}\d w\right)-\frac{1}{\ell}{H}(f+w)=\lambda,
\]
where we use that $\V \d w=Hw$, when we think of $\d w$ as a function on $SM^{1/n}$. It follows that the thermostat associated to a solution of the vortex equations is invariant under the action of the real part of the gauge group $\mathrm{G}_{\C}$. Therefore, without loosing generality, we can assume that $f$ vanishes identically, that is, $\hmet=\hmet_0$. Note however that the unitary part $\mathrm{G}$ does affect the associated thermostat. 
\end{Remark}

\section{Proof of Theorems \ref{mainthmintro:domsplit} and \ref{thm:anosovmain}}
Summarizing~\cref{sec:vortextothermo}, given a solution $(\hmet,\ov{\partial}_L,\varphi)$ to the vortex equations for a complex line bundle $L \to (M,g)$ and upon fixing an $n$-th root $SM^{1/n}$ of $SM$, we obtain a~\textit{vortex thermostat} on $SM^{1/n}$. After possibly applying a (non-unitary) gauge transformation to $(\hmet,\ov{\partial}_L,\varphi)$, we can assume that the thermostat $\phi$ arises from $\lambda=a-\V \theta$, where $a$ encodes a fractional differential on $M$, that is, a section $A$ of $K^{(m+n)/n}$ and $\theta$ a $1$-form on $M$ so that the following equations hold
\begin{equation}\label{eq:vortexvantheta}
K_g-\delta_g\theta=-1+\ell |A|^2_g \qquad \text{and}\qquad \ov{\partial}A=\ell\,\theta^{0,1}\otimes A,
\end{equation}
where for simplicity of notation we write $\ov{\partial}$ for $\ov{\partial}_{K^{(m+n)/n}}$ and where $\ell=m/n$. Thus, by \cref{lem:holcond} and \cref{lem:curvcond} our setup consists of $(X,H,\V)$ on $SM^{1/n}$ as well as real-valued functions $a,\theta$ satisfying $\V\V a=-(1+\ell)^2 a$ and $\V\V\theta=-\theta$ so that
\begin{align}
\label{eq:keyidcurv}K_g&=-1-X \theta-H\V\theta+\ell|\bm{A}|^2,\\
\label{eq:keyidhol}\frac{X\V a}{1+\ell}&=Ha+\frac{\ell\theta \V a}{1+\ell}-\ell a\V\theta,
\end{align}
where $\ell$ is a positive rational number and $\bm{A}=\frac{\V a}{1+\ell}+\imag a$. 

\subsection{Dominated splitting}

Applying~\cref{thm:workhorse} we obtain:

\begin{mainthm}\label{mainthm:domsplit}
Every vortex thermostat admits a dominated splitting.  More\-over, if all closed orbits of $\phi$ are hyperbolic saddles, then $\phi$ is Anosov.
\end{mainthm}
\begin{proof}
Using~\cref{thm:workhorse} we need to show that there exists a smooth function $p : SM^{1/n} \to \R$ so that
\[
\kappa_{p}=\kappa+Fp+p(p-\V\lambda)<0.
\]
Recall that $\lambda=a-\V\theta$. Taking $p=\theta+\V a/(1+\ell)$ we compute
\begin{align*}
\kappa_p-\kappa&=F\left(\theta+\frac{\V a}{1+\ell}\right)-\left(\theta+\frac{\V a}{1+\ell}\right)\left(\theta+\frac{\V a}{1+\ell}-\V a+\V\V\theta\right)\\
&=X\theta+Ha+\frac{\ell\theta \V a}{1+\ell}-\ell a\V\theta+\lambda \V p-\ell\left(\theta+\frac{\V a}{1+\ell}\right)\frac{\V a}{1+\ell}\\
&=X\theta+Ha-(1+\ell)a^2-(\V\theta)^2+2a\V\theta-\ell\left(\frac{\V a}{1+\ell}\right)^2\\
&=X\theta+Ha -\ell|\bm A|^2-a^2-(\V\theta)^2+2a\V\theta\\
&=-1-K_g-H\V\theta+H a-a^2-(\V\theta)^2+2a\V\theta\\
&=-1-(K_g-H\lambda+\lambda^2)=-1-\kappa
\end{align*}
where we have used that $\V\V\theta=-\theta$ and $\V\V a=-(1+\ell)^2 a$ as well as~\eqref{eq:defkappa}, \eqref{eq:keyidcurv} and \eqref{eq:keyidhol}. We conclude that $\kappa_p=-1$ and the existence of a dominated splitting follows. 

Finally, the addendum regarding the Anosov property when the closed orbits of $\phi$ are hyperbolic saddles is a consequence of \cite[Theorem B]{AR-H}. Indeed, in our situation the invariant normally hyperbolic irrational tori cannot arise since $V$ must be transversal to them. If we had one such torus $T$, then the projection map $\pi_{n}:SM^{1/n}\to M$ restricted to $T$ would be a local diffeomorphism which is absurd since $\chi(M)<0$.
\end{proof}

\subsection{The Anosov property} 

While we have an isomorphism $\mathcal{Z} : L \to K^{m/n}$ of complex line bundles, the two line bundles need not be isomorphic as holomorphic line bundles. We do however obtain:

\begin{mainthm}\label{thm:anosovmain}
Suppose $\mathcal{Z} : L \to K^{m/n}$ is an isomorphism of holomorphic line bundles, then the associated vortex thermostat is Anosov. 
\end{mainthm}
Recall from~\eqref{eq:difcplxstruc} that we write $\ov{\partial}_L=\ov{\partial}_{K^{m/n}}-\ell\,\theta^{0,1}$ for some $1$-form $\theta$ on $M$. The isomorphism $\mathcal{Z}$ being an isomorphism of holomorphic line bundles translates to $\theta$ vanishing identically. We thus henceforth restrict to the case where $\theta\equiv 0$, so that the equations~\eqref{eq:vortexvantheta} become
\begin{equation}\label{eq:coupvortexhol}
K_g=-1+\ell|A|^2_g,\quad \text{and}\quad \ov{\partial}A=0. 
\end{equation}

We start with the following comparison lemma:

\begin{Lemma} Let $\h$ be the positive Hopf solution of $Fh+\h^2+B\h-1=0$. Then
\[\frac{-c+\sqrt{c^2+4}}{2}\leqslant \h \leqslant \frac{c+\sqrt{c^2+4}}{2}\]
where $c=\max |B|$ and $B=\left(\frac{1-\ell}{1+\ell}\right)\V a$.
\label{lemma:fullbound}
\end{Lemma}

\begin{proof}[Proof of \cref{lemma:fullbound}]
We fix $(x,v) \in SM^{1/n}$. Recall from \cref{sec:dom+hyp} that the existence of a dominated splitting implies that the positive Hopf solution $h$ may be constructed using the limiting procedure
\[
h(x,v)=\lim_{R \to \infty}\eta_R(0),
\]
where for $R>0$ the function $\eta_R$ denotes the solution to the ODE
\[
\dot{\eta}(t)+\eta^2(t)+B(\phi_t(x,v))\eta(t)-1=0
\]
with $\eta_R(-R)=0$. Since $B\geqslant -c$ and $\h$ is positive, we have
\[\dot{\eta}=-\eta^2-B\eta+1\leqslant -\eta^2+c\eta+1.\] 
Hence if $\gamma$ solves the constant coefficients Riccati equation
\[\dot{\gamma}+\gamma^2-c\gamma-1=0\]
then $\eta(t)\leqslant \gamma(t)$ for $t\geqslant t_0$ provided $\eta(t_0)=\gamma(t_0)$ by ODE comparison. The solution $\gamma_R$ to $\dot{\gamma}+\gamma^2-c\gamma-1=0$ with $\gamma_{R}(-R)=0$ is given by
\[\gamma_{R}(t)=\frac{1-e^{(-R-t)/E}}{-C_{-}+C_{+}e^{(-R-t)/E}}\]
where 
\[C_{\pm}=\frac{c\pm\sqrt{c^2+4}}{2}\]
and $E=1/(C_{+}-C_{-})$. Thus
\[\eta_{R}(0)\leqslant \gamma_{R}(0)\to -1/C_{-}=C_{+}\]
as $R\to\infty$ and thus $h(x,v)\leqslant \frac{c+\sqrt{c^2+4}}{2}$.

The lower bound can also be proved in the same way. Since $B\leqslant c$, we have
\[\dot{\eta}=-\eta^2-B\eta+1\geqslant -\eta^2-c\eta+1.\] 
And now we compare with solutions of 
\[\dot{\gamma}+\gamma^2+c\gamma-1=0,\]
in particular those $\gamma_{R}$ with $\gamma_{R}(-R)=\infty$. One gets
\[\eta_{R}(0)\geqslant \gamma_{R}(0)\to \frac{-c+\sqrt{c^2+4}}{2}\]
as $R\to\infty$ and thus $h(x,v)\geqslant \frac{-c+\sqrt{c^2+4}}{2}$.
\end{proof}

For what follows we need a bound on $|A|_g^2$.

\begin{Lemma} Suppose $(g,A)$ satisfies $K_{g}=-1+\ell|A|^{2}_{g}$ and $\ov{\partial}A=0$.
Then $K_{g}<0$.
\label{lemma:boundA}
\end{Lemma}
In the case where $A$ is a differential of integral degree $d\geqslant 2$, the lemma was proved in~\cite{GabrielThomas_Thermo}. It is easy to check that the proof also holds in the case of a differential of fractional degree $d>1$. We refer the reader to~\cite[Lemma 5.2]{GabrielThomas_Thermo} for details. 

We are now ready to prove \cref{thm:anosovmain}.

\begin{proof}[Proof of \cref{thm:anosovmain}] We already know that the flow admits
a dominated splitting. To prove the Anosov property we shall use \cref{lemma:alternative}.
We will prove that in the range $\ell\geqslant 1$ our flows fit alternative (1) and for $0< \ell \leqslant 1$, they fit alternative (2). We shall prove the claims for the unstable bundle. The proofs for the stable bundle are quite analogous.

We note that \cref{lemma:boundA} gives
\begin{equation}
-1< \frac{\sqrt{\ell}}{1+\ell}\V a < 1.
\label{eq:boundA}
\end{equation}
Also note that for our thermostat $p=\V a/(1+\ell)$, $\kappa_{p}=-1$ and $h=r^u-p$.

Assume first that $\ell\geqslant 1$. We shall prove that $r^u>0$. This is equivalent to
\begin{equation}
h+\frac{\V a}{1+\ell}>0.
\label{eq:alt1}
\end{equation}
In view of \eqref{eq:boundA} and \eqref{eq:alt1} it is enough to prove that 
\[h\geqslant 1/\sqrt{\ell}.\]
From the definition of $c$ in \cref{lemma:fullbound} and the bound $\frac{\sqrt{\ell}}{1+\ell}\V a < 1$ we derive
$c\leqslant (1-\ell)/\sqrt{\ell}$. Hence
\[\frac{-c+\sqrt{c^2+4}}{2}\geqslant 1/\sqrt{\ell}\]
and the desired bound follows from \cref{lemma:fullbound}.

Assume now that $0<\ell\leqslant 1$. Condition (2) in \cref{lemma:alternative} for $r^u$ becomes
\begin{equation}
\left(\frac{\ell}{1+\ell}\right)\V a+1/h>0.
\label{eq:3.3}
\end{equation}
In view of \eqref{eq:boundA} and \eqref{eq:3.3} it is enough to prove that 
\[h\leqslant 1/\sqrt{\ell}.\]
From the definition of $c$ in \cref{lemma:fullbound} and the bound $\frac{\sqrt{\ell}}{1+\ell}\V a < 1$ we derive
$c\leqslant (1-\ell)/\sqrt{\ell}$. Hence
\[\frac{c+\sqrt{c^2+4}}{2}\leqslant 1/\sqrt{\ell}\]
and the desired bound follows from \cref{lemma:fullbound}.
\end{proof}

\begin{Remark}
As we have mentioned in the introduction, in the special case where $A$ is a cubic holomorphic differential, a solution $(g,A)$ to~\eqref{eq:coupvortexhol} gives rise to a properly convex projective structure on $M$. The monodromy representation of such a properly convex projective structure is an example of an~\textit{Anosov representation} as introduced by Labourie~\cite{MR2221137}. In recent work~\cite{MR4012341} Bochi, Potrie \& Sambarino show how Anosov representations can be used to construct certain cocycles admitting a dominated splitting. At the time of writing, it is however quite unclear if there is any relation between~\cite{MR4012341} and our construction which goes beyond the special case of cubic holomorphic differentials.  
\end{Remark}

\section{Examples}

Let $M$ be a closed oriented surface equipped with a hyperbolic metric $g_0$. Assume furthermore that the unit tangent bundle $SM$ of $(M,g_0)$ admits an $n$-th root $SM^{1/n}$, so that correspondingly we have an $n$-th root $K^{1/n}$ of the canonical bundle $K$ of $(M,g_0)$. Let $m$ be a positive integer and write $\ell=m/n$. We equip $K^{1+\ell}$ with the holomorphic structure determined by $g_0$, that is, in our previous notation, we choose $\theta\equiv 0$. Suppose $A$ is a holomorphic differential of fractional degree $1+\ell$. Note that such differentials exist by the Riemann--Roch theorem. In order to obtain one of our Anosov flows, we must thus find a metric $g$ in the conformal equivalence class of $g_0$ so that
\[
K_g=-1+\ell |A|^2_g.
\]
Under a conformal change $g_0 \mapsto \mathrm{e}^{2u}g_0$ with $u\in C^{\infty}(M)$, the norm $|A|_{g_0}^2$ changes as 
\[
|A|^2_{\mathrm{e}^{2u}g_0}=\mathrm{e}^{-2(1+\ell)u}|A|^2_{g_0}.
\]
We also have the identity 
\[
K_{\mathrm{e}^{2u}g_0}=\mathrm{e}^{-2u}(-1-\Delta u)
\]
for the change of the Gauss curvature under conformal change. Here $\Delta$ denotes the Laplace operator with respect to the hyperbolic metric. Writing $g=\mathrm{e}^{2u}g_0$, we thus obtain the PDE
\[
\Delta u=-1 +\mathrm{e}^{2u}-\ell\mathrm{e}^{-2\ell u}\alpha
\]
with $\alpha:=|A|^2_{g_0}$. Since $\alpha\geqslant 0$, this quasi-linear elliptic PDE admits a unique smooth solution which can be obtained by standard methods, see for instance~\cite[Prop.~1.9]{MR2744149}. Therefore, we obtain a solution to the vortex equations and an associated Anosov flow.

\begin{Remark}
Recall that every closed oriented hyperbolic Riemann surface $(M,g_0)$ admits a~\textit{Fuchsian model}, which realises its unit tangent bundle $SM$ as a quotient $\Gamma\setminus \mathrm{PSL}(2,\R)$, where $\Gamma\subset \mathrm{PSL}(2,\R)$ is a~\textit{Fuchsian group}, that is, a discrete torsion-free subgroup of $\mathrm{PSL}(2,\R)$. Therefore, we obtain a square root $SM^{1/2}\simeq \tilde{\Gamma}\setminus \mathrm{SL}(2,\R)$, where $\tilde{\Gamma}\subset \mathrm{SL}(2,\R)$ denotes the preimage of $\Gamma$ under the $2$-fold cover $\mathrm{SL}(2,\R) \to \mathrm{PSL}(2,\R)$. Since the unit tangent bundles with respect to conformally equivalent metrics are isomorphic as principal $\mathrm{SO}(2)$-bundles, we also obtain a square root of the unit tangent bundle for every metric in the conformal equivalence class of $g_0$. In particular, on every closed hyperbolic Riemann surface we obtain an Anosov flow on $SM^{1/2}$ from a holomorphic differential $A$ of fractional degree $1+1/2=3/2$. These flows are topologically orbit equivalent to the lift of a constant curvature geodesic flow \cite{MR758894}, but do not arise from the lift of a flow on $SM$.
\end{Remark} 

\appendix

\section{Variants of the vortex equations}\label{app:gencase}

Instead of our variant of the vortex equations, we may also consider the following pair of equations on an oriented Riemannian $2$-manifold $(M,g)$ of negative Euler characteristic 
\begin{equation}\label{eq:genvortex}
K_g-\delta_g\theta=-1+\ell\mathrm{e}^{2f}|A|^2_g-\frac{1}{k}\Delta_g f \qquad\text{and} \qquad \ov{\partial} A=k\, \theta^{0,1}\otimes A. 
\end{equation}
Here $A$ is a differential of fractional degree $1+\ell>1$, $\theta \in \Omega^1$, $f\in C^{\infty}$ and $k$ is a real constant. Notice that we recover our vortex equations by choosing $k=\ell$. We leave it as an exercise to the interested reader to check that for the choice $c=2(\ell+1)$, the usual vortex equations~\eqref{eq:usualvortex} are equivalent to~\eqref{eq:genvortex} when $k=\ell+1$. Again, it is straightforward to verify that~\eqref{eq:genvortex} are invariant under suitable gauge transformations. Namely, writing a gauge transformation as $\tau=\mathrm{e}^{w+\imag \vartheta}$ for $w,\vartheta\in C^{\infty}$, we obtain a solution 
\[
\tau \cdot (A,\theta,f)=\left(\mathrm{e}^{-(w+\imag\vartheta)}A,\theta-\frac{1}{k}(\d w +\star_g\d\vartheta),f+w\right)
\]
to the above vortex equations from a solution $(A,\theta,f)$. As before, we obtain a thermostat on a suitable root $SM^{1/n}$ of $SM$, by defining
\[
\lambda=\mathrm{e}^f a-\V\theta-\frac{1}{\ell}Hf. 
\]
where we use notation as in~\cref{sec:vortextothermo}. The thermostat is again invariant under real gauge transformations of the form $\tau=\mathrm{e}^w$, so that we can assume that $f$ vanishes identically. Thus we have
\[
K_g-\delta_g\theta=-1+\ell|A|^2_g\qquad \text{and}\qquad \ov{\partial}A=k\, \theta^{0,1}\otimes A. 
\]
Taking $p=\theta+\V a/(1+\ell)$, we compute exactly as in the proof of \cref{mainthm:domsplit} that
\[
\kappa_p=-1+(k-\ell)\Re\left((\theta+\imag\V\theta)\bm A\right)
\]
where $\bm{A}=\frac{\V a}{1+\ell}+\imag a$. For the usual vortex equations with $k=\ell+1$ we thus obtain $\kappa_p=-1+\Re\left((\theta+\imag\V\theta)\bm A\right)$. Moreover, for the usual vortex equations we have the bound $|\bm A|^2\leqslant 1/\ell$, see~\cite[Prop.~5.2]{MR1086749}. Thus, we still obtain a dominated splitting provided $|\theta+\imag\V\theta|<\sqrt{\ell}$. 

\begin{Remark}
We do not know if we still obtain a dominated splitting if the bound $|\theta+\imag\V\theta|<\sqrt{\ell}$ does not hold. 
\end{Remark}

\providecommand{\mr}[1]{\href{http://www.ams.org/mathscinet-getitem?mr=#1}{MR~#1}}
\providecommand{\zbl}[1]{\href{http://www.zentralblatt-math.org/zmath/en/search/?q=an:#1}{zbM~#1}}
\providecommand{\arxiv}[1]{\href{http://www.arxiv.org/abs/#1}{arXiv:#1}}
\providecommand{\doi}[1]{\href{http://dx.doi.org/#1}{DOI~#1}}
\providecommand{\href}[2]{#2}


\begin{thebibliography}{}

\bibitem{AR-H}
\bgroup\scshape{}A.~Arroyo\egroup{},
  \bgroup\scshape{}F.~Rodriguez~Hertz\egroup{}, Homoclinic bifurcations and
  uniform hyperbolicity for three-dimensional flows,  \emph{Ann. Inst. H.
  Poincar\'{e} Anal. Non Lin\'{e}aire} \textbf{20} (2003), 805--841.
  \doi{10.1016/S0294-1449(03)00016-7}  \mr{1995503}\;  \zbl{1045.37006}\;

\bibitem{MR2766227}
\bgroup\scshape{}M.~Asaoka\egroup{}, Regular projectively {A}nosov flows on
  three-dimensional manifolds,  \emph{Ann. Inst. Fourier (Grenoble)}
  \textbf{60} (2010), 1649--1684. \doi{10.5802/aif.2569}  \mr{2766227}\;
  \zbl{1202.37030}\;

\bibitem{MR2094116}
\bgroup\scshape{}Y.~Benoist\egroup{}, Convexes divisibles. {I},  in
  \emph{Algebraic groups and arithmetic}, Tata Inst. Fund. Res., Mumbai, 2004,
  pp.~339--374. \mr{2094116}\;  \zbl{1084.37026}\;

\bibitem{MR2529495}
\bgroup\scshape{}J.~Bochi\egroup{}, \bgroup\scshape{}N.~Gourmelon\egroup{},
  Some characterizations of domination,  \emph{Math. Z.} \textbf{263} (2009),
  221--231. \doi{10.1007/s00209-009-0494-y}  \mr{2529495}\;  \zbl{1181.37032}\;

\bibitem{MR4012341}
\bgroup\scshape{}J.~Bochi\egroup{}, \bgroup\scshape{}R.~Potrie\egroup{},
  \bgroup\scshape{}A.~Sambarino\egroup{}, Anosov representations and dominated
  splittings,  \emph{J. Eur. Math. Soc. (JEMS)} \textbf{21} (2019), 3343--3414.
  \doi{10.4171/JEMS/905}  \mr{4012341}\;  \zbl{1429.22011}\;

\bibitem{MR2105774}
\bgroup\scshape{}C.~Bonatti\egroup{}, \bgroup\scshape{}L.~J.
  D\'{\i}az\egroup{}, \bgroup\scshape{}M.~Viana\egroup{}, \emph{Dynamics beyond
  uniform hyperbolicity}, \emph{Encyclopaedia of Mathematical Sciences}
  \textbf{102}, Springer-Verlag, Berlin, 2005, A global geometric and
  probabilistic perspective, Mathematical Physics, III. \mr{2105774}\;
  \zbl{1060.37020}\;

\bibitem{MR1086749}
\bgroup\scshape{}S.~B. Bradlow\egroup{}, Vortices in holomorphic line bundles
  over closed {K}\"{a}hler manifolds,  \emph{Comm. Math. Phys.} \textbf{135}
  (1990), 1--17. \mr{1086749}\;  \zbl{0717.53075}\;

\bibitem{Bry}
\bgroup\scshape{}R.~L. Bryant\egroup{}, Projectively flat {F}insler
  {$2$}-spheres of constant curvature,  \emph{Selecta Math.~(N.S.)} \textbf{3}
  (1997), 161--203. \doi{10.1007/s000290050009}  \mr{1466165}\;
  \zbl{0897.53052}\;

\bibitem{MR1959059}
\bgroup\scshape{}K.~Cieliebak\egroup{}, \bgroup\scshape{}A.~R. Gaio\egroup{},
  \bgroup\scshape{}I.~Mundet~i Riera\egroup{}, \bgroup\scshape{}D.~A.
  Salamon\egroup{}, The symplectic vortex equations and invariants of
  {H}amiltonian group actions,  \emph{J. Symplectic Geom.} \textbf{1} (2002),
  543--645. \mr{1959059}\;  \zbl{1093.53093}\;

\bibitem{CP2015}
\bgroup\scshape{}S.~Crovisier\egroup{}, \bgroup\scshape{}R.~Potrie\egroup{},
  Introduction to partially hyperbolic dynamics,  in \emph{School on Dynamical
  Systems}, International Centre for Theoretical Physics, Trieste, 2015.

\bibitem{MR2911018}
\bgroup\scshape{}M.~Dunajski\egroup{}, Abelian vortices from sinh-{G}ordon and
  {T}zitzeica equations,  \emph{Phys. Lett. B} \textbf{710} (2012), 236--239.
  \doi{10.1016/j.physletb.2012.02.078}  \mr{2911018}\;

\bibitem{EM98}
\bgroup\scshape{}Y.~M. Eliashberg\egroup{}, \bgroup\scshape{}W.~P.
  Thurston\egroup{}, \emph{Confoliations}, \emph{University Lecture Series}
  \textbf{13}, American Mathematical Society, Providence, RI, 1998.
  \doi{10.1090/ulect/013}  \mr{1483314}\;  \zbl{0893.53001}\;

\bibitem{MR1812682}
\bgroup\scshape{}G.~Gallavotti\egroup{}, New methods in nonequilibrium gases
  and fluids,  \emph{Open Syst. Inf. Dyn.} \textbf{6} (1999), 101--136.
  \doi{10.1023/A:1009696806769}  \mr{1812682}\;  \zbl{1044.82548}\;

\bibitem{MR1489572}
\bgroup\scshape{}G.~Gallavotti\egroup{}, \bgroup\scshape{}D.~Ruelle\egroup{},
  S{RB} states and nonequilibrium statistical mechanics close to equilibrium,
  \emph{Comm. Math. Phys.} \textbf{190} (1997), 279--285.
  \doi{10.1007/s002200050241}  \mr{1489572}\;  \zbl{0909.60091}\;

\bibitem{MR2945757}
\bgroup\scshape{}H.~Geiges\egroup{},
  \bgroup\scshape{}J.~Gonzalo~P\'{e}rez\egroup{}, Generalised spin structures
  on 2-dimensional orbifolds,  \emph{Osaka J. Math.} \textbf{49} (2012),
  449--470. \mr{2945757}\;  \zbl{1260.57044}\;

\bibitem{MR758894}
\bgroup\scshape{}E.~Ghys\egroup{}, Flots d'{A}nosov sur les
  {$3$}-vari\'{e}t\'{e}s fibr\'{e}es en cercles,  \emph{Ergodic Theory Dynam.
  Systems} \textbf{4} (1984), 67--80. \doi{10.1017/S0143385700002273}
  \mr{758894}\;  \zbl{0527.58030}\;

\bibitem{MR887284}
\bgroup\scshape{}N.~Hitchin\egroup{}, The self-duality equations on a {R}iemann
  surface,  \emph{Proc. London Math. Soc. (3)} \textbf{55} (1987), 59--126.
  \doi{10.1112/plms/s3-55.1.59}  \mr{887284}\;  \zbl{0634.53045}\;

\bibitem{Hopf48}
\bgroup\scshape{}E.~Hopf\egroup{}, Closed surfaces without conjugate points,
  \emph{Proc. Nat. Acad. Sci. U.S.A.} \textbf{34} (1948), 47--51.
  \doi{10.1073/pnas.34.2.47}  \mr{23591}\;  \zbl{0030.07901}\;

\bibitem{MR614447}
\bgroup\scshape{}A.~Jaffe\egroup{}, \bgroup\scshape{}C.~Taubes\egroup{},
  \emph{Vortices and monopoles}, \emph{Progress in Physics} \textbf{2},
  Birkh\"{a}user, Boston, Mass., 1980, Structure of static gauge theories.
  \mr{614447}\;  \zbl{0457.53034}\;

\bibitem{MR3329885}
\bgroup\scshape{}I.~Kim\egroup{}, \bgroup\scshape{}A.~Papadopoulos\egroup{},
  Convex real projective structures and {H}ilbert metrics,  in \emph{Handbook
  of {H}ilbert geometry}, \emph{IRMA Lect. Math. Theor. Phys.} \textbf{22},
  Eur. Math. Soc., Z\"urich, 2014, pp.~307--338. \mr{3329885}\;

\bibitem{MR2221137}
\bgroup\scshape{}F.~Labourie\egroup{}, Anosov flows, surface groups and curves
  in projective space,  \emph{Invent. Math.} \textbf{165} (2006), 51--114.
  \doi{10.1007/s00222-005-0487-3}  \mr{2221137}\;  \zbl{1103.32007}\;

\bibitem{MR2402597}
\bgroup\scshape{}F.~Labourie\egroup{}, Flat projective structures on surfaces
  and cubic holomorphic differentials,  \emph{Pure Appl. Math. Q.} \textbf{3}
  (2007), 1057--1099. \mr{2402597}\;  \zbl{1158.32006}\;

\bibitem{MR1828223}
\bgroup\scshape{}J.~C. Loftin\egroup{}, Affine spheres and convex
  {$\mathbb{RP}^n$}-manifolds,  \emph{Amer. J. Math.} \textbf{123} (2001),
  255--274. \mr{1828223}\;  \zbl{0997.53010}\;

\bibitem{MR3987443}
\bgroup\scshape{}T.~Mettler\egroup{}, Minimal {L}agrangian connections on
  compact surfaces,  \emph{Adv. Math.} \textbf{354} (2019), {A}rticle {ID}
  106747, 36 pp. \doi{10.1016/j.aim.2019.106747}  \mr{3987443}\;
  \zbl{07103867}\;

\bibitem{GabrielThomas_Thermo}
\bgroup\scshape{}T.~Mettler\egroup{}, \bgroup\scshape{}G.~P.
  Paternain\egroup{}, Holomorphic differentials, thermostats and {A}nosov
  flows,  \emph{Math. Ann.} \textbf{373} (2019), 553--580.
  \doi{10.1007/s00208-018-1712-x}  \mr{3968880}\;  \zbl{1415.53051}\;

\bibitem{MR4109900}
\bgroup\scshape{}T.~Mettler\egroup{}, \bgroup\scshape{}G.~P.
  Paternain\egroup{}, Convex projective surfaces with compatible {W}eyl
  connection are hyperbolic,  \emph{Anal. PDE} \textbf{13} (2020), 1073--1097.
  \doi{10.2140/apde.2020.13.1073}  \mr{4109900}\;  \zbl{07221197}\;

\bibitem{Mit95}
\bgroup\scshape{}Y.~Mitsumatsu\egroup{}, Anosov flows and non-{S}tein
  symplectic manifolds,  \emph{Ann. Inst. Fourier (Grenoble)} \textbf{45}
  (1995), 1407--1421. \mr{1370752}\;  \zbl{0834.53031}\;

\bibitem{MR907998}
\bgroup\scshape{}M.~Noguchi\egroup{}, Yang-{M}ills-{H}iggs theory on a compact
  {R}iemann surface,  \emph{J. Math. Phys.} \textbf{28} (1987), 2343--2346.
  \doi{10.1063/1.527769}  \mr{907998}\;  \zbl{0641.53082}\;

\bibitem{MR1705592}
\bgroup\scshape{}D.~Ruelle\egroup{}, Smooth dynamics and new theoretical ideas
  in nonequilibrium statistical mechanics,  \emph{J. Statist. Phys.}
  \textbf{95} (1999), 393--468. \doi{10.1023/A:1004593915069}  \mr{1705592}\;
  \zbl{0934.37010}\;

\bibitem{MR3457601}
\bgroup\scshape{}M.~Sambarino\egroup{}, A (short) survey on dominated
  splittings,  in \emph{Mathematical {C}ongress of the {A}mericas},
  \emph{Contemp. Math.} \textbf{656}, Amer. Math. Soc., Providence, RI, 2016,
  pp.~149--183. \doi{10.1090/conm/656/13105}  \mr{3457601}\;
  \zbl{1378.37066}\;

\bibitem{MR228014}
\bgroup\scshape{}S.~Smale\egroup{}, Differentiable dynamical systems,
  \emph{Bull. Amer. Math. Soc.} \textbf{73} (1967), 747--817.
  \doi{10.1090/S0002-9904-1967-11798-1}  \mr{228014}\;  \zbl{0202.55202}\;

\bibitem{MR2744149}
\bgroup\scshape{}M.~E. Taylor\egroup{}, \emph{Partial differential equations
  {III}. {N}onlinear equations}, second ed., \emph{Applied Mathematical
  Sciences} \textbf{117}, Springer, New York, 2011.
  \doi{10.1007/978-1-4419-7049-7}  \mr{2744149}\;  \zbl{1206.35004}\;

\bibitem{W1}
\bgroup\scshape{}M.~P. Wojtkowski\egroup{}, Magnetic flows and {G}aussian
  thermostats on manifolds of negative curvature,  \emph{Fund. Math.}
  \textbf{163} (2000), 177--191. \doi{10.4064/fm-163-2-177-191}  \mr{1752103}\;
   \zbl{0997.37011}\;

\bibitem{W2}
\bgroup\scshape{}M.~P. Wojtkowski\egroup{}, {$W$}-flows on {W}eyl manifolds and
  {G}aussian thermostats,  \emph{J. Math. Pures Appl. (9)} \textbf{79} (2000),
  953--974. \doi{10.1016/S0021-7824(00)00176-8}  \mr{1801870}\;
  \zbl{0969.37017}\;

\bibitem{W3}
\bgroup\scshape{}M.~P. Wojtkowski\egroup{}, Monotonicity,
  {$\mathcal{J}$}-algebra of {P}otapov and {L}yapunov exponents,  in
  \emph{Smooth ergodic theory and its applications ({S}eattle, {WA}, 1999)},
  \emph{Proc. Sympos. Pure Math.} \textbf{69}, Amer. Math. Soc., Providence,
  RI, 2001, pp.~499--521. \doi{10.1090/pspum/069/1858544}  \mr{1858544}\;
  \zbl{1013.37028}\;

\end{thebibliography}
\end{document}